\documentclass[sn-mathphys]{sn-jnl}

\jyear{2021}%

\theoremstyle{thmstyleone}%
\newtheorem{theorem}{Theorem}
\theoremstyle{thmstyletwo}%
\newtheorem{remark}{Remark}%

\theoremstyle{thmstylethree}%
\newtheorem{lemma}{Lemma}
\begin{document}

\title[Stability of a porous system with Gurtin-Pipkin law]{On the time behavior of a porous thermoelastic system with
	only thermal dissipation given by Gurtin-Pipkin law}


\author[1]{\fnm{Afaf} \sur{Ahmima}}\email{afafahmima@gmail.com}

\author*[1]{\fnm{Abdelfeteh} \sur{Fareh}}\email{farehabdelf@gmail.com}
\equalcont{These authors contributed equally to this work.}

\affil[1]{\orgdiv{Laboratory of operators theory and PDEs: foundations and applications}, \orgname{University of El Oued}, \orgaddress{\street{P.O.B. 789}, \city{El Oued}, \postcode{39000},  \country{Algeria}}}


\abstract{In the present paper we consider a porous thermoelastic system with only one dissipative mechanism generated by the heat conductivity modelled by the Gurtin-Pipkin thermal law.
	By the use of a semigroup approach and the Lumer-Phillips theorem we prove the existence of a unique solution. We introduce a stability number $\chi_g$ depends on the coefficients of the system, and establish an exponential stability result provided that  $\chi_g=0$. Otherwise, if $\chi_g\ne 0$, we prove the lack of exponential decay. Our result improves and generalizes the previous results in the literature obtained for Fourier's and Cattaneo's laws of thermal conductivity.
}




\maketitle

\section{Introduction}

In this paper we are concerned by the following porous thermelastic system%
\begin{equation}
	\left\{ 
	\begin{array}{ll}
		\rho u_{tt}=au_{xx}+b\phi _{x} & {in}~~ (0,\pi )\times(0,\infty ), \\ 
		J\phi _{tt}=\alpha \phi _{xx}-bu_{x}-\xi \phi -\beta \theta _{x} & {in}%
		~~(0,\pi )\times (0,\infty ), \\ 
		c\theta _{t}=-q_{x}-\beta \phi _{xt} & {in}~~(0,\pi )\times(0,\infty ), \\ 
		q=-\displaystyle\int_{-\infty }^{0}g(t-s)\theta (x,s)ds & {in}~~(0,\pi )\times(0,\infty ).%
	\end{array}%
	\right.  \label{A1}
\end{equation}%
This system models a one-dimensional porous elastic material of length $\pi $%
, subjected to thermal effects modelled by the Gurtin-Pipkin law of thermal
conductivity. Here, $u$ is the transversal displacement, $\phi $ is the volume
fraction, $\theta $ is the difference of temperature from an equilibrium
reference value and $q$ is the heat flux. The coefficients $\rho ,J,c,\mu
,b,\alpha ,\xi $ are positive constitutive constants. Moreover, to ensure that the energy associated to the solution of \eqref{A1} is a positive definite form, we assume that $b,\mu,\xi$ satisfy $\mu \xi
>b^{2}$. The coefficient $\beta $ is a coupling constant that is different
from zero but its sign does not matter in the analysis.

 The basic evolution equations in the theory of thermoelastic materials with voids are developed by Ie\c{s}an \cite{Iesan86,Iesan}. In the one dimensional case these equations are written as follows 
\begin{equation}\label{aa}
	\rho u_{tt}=T_x, ~~ J\varphi_{tt}=H_x-G,~~ \rho\eta_{t}=q_x,
\end{equation}
where, $T$ is the stress tensor, $H$ is the equilibrated stress vector, $G$ is the equilibrated body force and $\eta$ is the entropy.

In the linear theory, and assuming that the body is free from stresses and has zero intrinsic equilibrated body force, the constitutive equations are \cite{Iesan86,Iesan}
\begin{equation}\label{bb}
	\begin{array}{ll}
		T=\mu u_{x}+b\varphi,&
		H=\alpha\varphi_{x}-\beta\theta,\\
		G=bu_{x}+\xi\varphi,&
		\rho\eta=-c\theta-\beta\varphi_{x}.
	\end{array}
\end{equation}
The substitution of the constitutive equation \eqref{bb} in the evolution equations \eqref{aa} leads to the first three equation of the system \eqref{A1}. The forth equation of \eqref{A1} represents the heat conduction modelled by Gurtin-Pipkin thermal law \eqref{1b} which will be introduced later. 

Much has been written on the study of the longtime behavior of porous
thermoelastic systems. Quintanilla \cite{Quin1} proved the slow decay of the
solution of the system 
\begin{equation}
	\left\{ 
	\begin{array}{l}\label{A}
		\rho u_{tt}=\mu u_{xx}+\beta \varphi _{x}, \\ 
		\rho \kappa \varphi _{tt}=\alpha \varphi _{xx}-\beta u_{x}-\xi \varphi -\tau
		\varphi _{t},%
	\end{array}%
	\right.  
\end{equation}%
subjected to the Dirichlet-Neumann or the Neumann-Dirichlet boundary conditions.\\
Later in 2017, Apalara \cite{Apalara} and, independently, Santos $\emph{et}$ \emph{al.} \cite{Santos1} established  the exponential stability of (\ref{A})
provided that the wave speeds of the equations of (\ref{1a}) are equal%
\begin{equation}
	\frac{\mu }{\rho }=\frac{\alpha }{\rho \kappa }.  \label{a1}
\end{equation}
The same result was obtained by Santos \emph{et al.} \cite{Santos1}
when the damping $-\tau \varphi _{t}$ is replaced by the elastic damping $\gamma u_{t}$ in the first equation.

During the last three decades, several dissipative mechanisms have been examined to exponentially stabilize the solution of (\ref{A}). First, Maga\~na and Quintanilla replace the visco-porous damping $-\tau\varphi_{t}$ by a visco-elastic one of the form $\lambda u_{txx}$ in the first equation of \eqref{A} and proved the non-exponential decay of the solution. Casas and Quintanilla \cite{Casas1}, coupled the system \eqref{A} with the heat equation. They examined the system
\begin{equation}\label{AB}
	\left\{\begin{array}{l}
		\rho u_{tt}=\mu u_{xx}+b\varphi_x-\beta\theta_x,\\
		J\varphi_{tt}=\alpha\varphi_{xx}-bu_x-\xi\varphi+m\theta-\tau\varphi_{t},\\
		c\theta_{t}=k^{\star}\theta_{xx}-\beta u_{xt}-m\varphi_{t},
	\end{array}\right.
\end{equation}
with the boundary conditions
\[u(0,t)=u(\pi,t)=\varphi_{x}(0,t)=\varphi_{x}(\pi,t)=\theta_{x}(0,t)=\theta_{x}(\pi,t)=0,\]
and proved that the combination of the visco-porous and thermal dissipations lead to an exponential rate of decay. The same rate of decay was obtained when micro-temperature and viscoelastic dissipations are combined \cite{Magana}. However, removing the visco-porous damping $-\tau\varphi_{t}$ from the system \eqref{AB}, or replacing it by the viscoelastic damping $\lambda u_{txx}$ in the first equation of \eqref{AB}, lead to a lose of the exponential decay \cite{Casas2,Magana}.

Apalara \cite{Apalara1} and Feng \cite{Feng1} examined, respectively, the effect of the dissipation introduced by a memory term or a past history. They obtained a general rate of decay depending on the rate of the relaxation function.

Note that Magaña and Quintanilla \cite{Magana} are summarized the effects of the above dissipative mechanisms as follows:\\
if a thermal or a viscoelastic damping is combined with a micro-thermal or a visco-porous damping, the solution decays exponentially, otherwise, the solution decays in a slow way.

The rate of decay of Timoshenko type systems is quite similar to that of porous thermoelastic systems.  We cite here some of them. In a pioneer work, Soufyane \cite{Soufyane} considered the Timoshenko system%
\begin{equation}
	\left\{ 
	\begin{array}{l}
		\rho _{1}\varphi _{tt}=\kappa (\varphi _{x}+\psi )_{x}, \\ 
		\rho _{2}\psi _{tt}=\alpha \varphi _{xx}-\kappa (\varphi _{x}+\psi
		)-b(x)\psi _{t},%
	\end{array}%
	\right.  \label{B}
\end{equation}%
where $b\in L^{\infty }(0,L)$ is a strictly positive function, and showed
that the system is exponentially stable if and only if 
\begin{equation}
	\frac{\kappa }{\rho _{1}}-\frac{\alpha }{\rho _{2}}=0.  \label{C}
\end{equation}

The strong damping $b(x)\psi_{t}$ was replaced by the memory term $\int_{0}^{t}g(t-s)\psi_{xx}(x,s)ds$ in \cite{Ammar,Mustafa,Tatar} and an
exponential stability result was obtained for an exponentially decaying
relaxation function $g$.

Rivera and Racke \cite{Racke} examined the heat dissipation produced by the coupling with the heat equation. They considered the system 
\begin{equation}
	\left\{ 
	\begin{array}{l}
		\rho _{1}\varphi _{tt}=\kappa (\varphi _{x}+\psi )_{x}, \\ 
		\rho _{2}\psi _{tt}=b\varphi _{xx}-\kappa (\varphi _{x}+\psi )-\gamma \theta
		_{x}, \\ 
		\rho _{3}\theta _{t}=k\theta _{xx}-\gamma \psi _{tx},%
	\end{array}%
	\right.  \label{D}
\end{equation}%
with different boundary conditions, and proved that the solution decays exponentially, if and
only if (\ref{C}) holds. The same result was obtained by Almeida Junior 
\emph{et al.} \cite{Almeida} for the following Timoshenko type system with a
thermoelastic coupling in shear force 
\begin{equation}
	\left\{ 
	\begin{array}{ll}
		\rho _{1}\varphi _{tt}-\kappa \left( \varphi _{x}+\psi \right) _{x}+\sigma
		\theta _{x}=0, & in\;\left] 0,L\right[ \times \mathbb{R}_{+}, \\ 
		\rho _{2}\psi _{tt}-b\psi _{xx}+\kappa \left( \varphi _{x}+\psi \right)
		-\sigma \theta =0, & in\;\left] 0,L\right[ \times \mathbb{R}_{+}, \\ 
		\rho _{3}\theta _{t}-\gamma \theta _{xx}+\sigma \left( \varphi _{x}+\psi
		\right) _{t}=0, & in\;\left] 0,L\right[ \times \mathbb{R}_{+}.%
	\end{array}%
	\right.  \label{E}
\end{equation}

We can summarize the rate of decay of Timosenko type systems is as follows: \begin{itemize}
	\item [$\cdot$]
	Two dissipative mechanisms lead, always, to an exponential stability result regardless any assumption on the coefficients of the system \cite{Kim,Raposo,Shi}.
	\item[$\cdot$] A single dissipation mechanism  produce an exponential rate of decay, if and only if the wave speeds of the hyperbolic part of the system are equal.
\end{itemize}
Note that the heat conduction in (\ref{AB}) and (\ref{E}) is given through 
Fourier's law 
\begin{equation}
	q=-\kappa \theta _{x},  \label{a2}
\end{equation}%
which leads to a parabolic equation and consequently, the heat propagates
with an infinite speed, which is unrealistic.  To overcome this physical paradox, many alternative
theories were developed. Lord and Shulman \cite{Lord} replaced the Fourier
law \eqref{a2} by the Cattaneo one 
\begin{equation}
	\tau _{0}q_{t}+q+\kappa \theta _{x}=0,  \label{d}
\end{equation}%
where the positive constant $\tau _{0}$ represents the time lag in the
response of the heat flux to the temperature gradient. According to this
theory, (called second sound thermoelasticity), the system becomes fully
hyperbolic and the heat propagates as a wave with a finite speed. Moreover, the speed of the heat equation is involved in the exponential stability condition. 

In the context of this theory, Messaoudi and Fareh \cite{Fareh5} considered
the porous thermoelastic system of second sound 
\begin{equation}
	\left\{ 
	\begin{array}{l}
		\rho u_{tt}=\mu u_{xx}+b\phi _{x}, \\ 
		J\phi _{tt}=\alpha \phi _{xx}-bu_{x}-\xi \phi -\beta \theta _{x}, \\ 
		c\theta _{t}=-q_{x}-\beta \phi _{tx}-\delta \theta , \\ 
		\tau _{0}q_{t}+q+\kappa \theta _{x}=0,%
	\end{array}%
	\right.  \label{b}
\end{equation}%
where $\rho ,\mu ,J,\alpha ,\xi ,\beta ,c,\kappa $ and $\tau _{0}$ are
positive constants and $\mu \xi =b^{2}$. They introduced the stability
number 
\begin{equation*}
	\chi =\beta ^{2}-\left( \frac{c\alpha \mu }{\rho }-\frac{\alpha \kappa }{%
		\tau _{0}}\right) \left( \frac{J}{\alpha }-\frac{\rho }{\mu }\right)
\end{equation*}%
and proved that the solution of (\ref{b}) is exponentially stable if and
only if $\chi =0$.

To the best of our knowledge \cite{Fareh5} is the only contribution that introduces
the stability number for porous thermoelastic problems. However, for
Timoshenko type systems, the stability number was introduced first by Santos 
\emph{et al.} \cite{Santos} for Timoshenko systems. They considered the following system with second sound thermoelasticity 
\begin{equation}
	\left\{ 
	\begin{array}{ll}
		\rho _{1}\varphi _{tt}-\kappa (\varphi _{x}+\psi )_{x}=0, & ~\mbox{in}%
		~(0,l)\times \mathbb{R}_{+}, \\ 
		\rho _{2}\psi _{tt}-b\psi _{xx}+\kappa (\varphi _{x}+\psi )+\delta \theta
		_{x}=0, & ~\mbox{in}~(0,l)\times \mathbb{R}_{+}, \\ 
		\rho _{3}\theta _{t}+q_{x}+\delta \psi _{tx}=0, & ~\mbox{in}~(0,l)\times 
		\mathbb{R}_{+}, \\ 
		\tau q_{t}+\beta q+\theta _{x}=0, & ~\mbox{in}~(0,l)\times \mathbb{R}_{+},%
	\end{array}%
	\right.  \label{b1}
\end{equation}%
and showed the exponential stability of the solution of (\ref{b1}), if and
only if 
\begin{equation}
	\chi _{0}=\left( \tau -\frac{\rho _{1}}{\rho _{3}\kappa }\right) \left( \rho
	_{2}-\frac{b\rho _{1}}{\kappa }\right) -\frac{\tau \rho _{1}\delta ^{2}}{%
		\rho _{3}\kappa }=0.
\end{equation}%
In fact, the stability number $\chi $ defined in (\ref{C}) is recovered from 
$\chi _{0}$ at the limit $\tau \rightarrow 0,$ when (\ref{b1}) collapses
to (\ref{D}).

It is worth noting that the theory of thermoelasticity with second sound is incapable to depict the memory effect which prevails in some materials, particularly at low temperature. This fact leads to the search for a more general constitutive hypothesis that links the heat flux to the thermal memory. Gurtin and Pipkin \cite{Gurtin} supposed that the heat flux depends on the accumulated  history of the temperature gradient weighted by a relaxation
function called the heat flux kernel. They formulated a general nonlinear
theory for which thermal perturbations disseminate with finite speed. According to this theory, the linearized constitutive equation for $q$ writes  
\begin{equation}
	q=-\int_{-\infty}^{t}g\left( t-s\right) \theta_{x}\left( x,s\right) ds,
	\label{1b}
\end{equation}
where $g(s)$ is the heat conductivity relaxation kernel. The presence of
convolution term (\ref{1b}) in the heat equation renders the system into a fully hyperbolic system, which allows the heat to propagate
with finite speed and accepts to describe the memory effect of heat
conduction. We notice that several constitutive models arise from
different choices of $g\left( s\right) .$ In particular, if we take $%
g\left( s\right) =\kappa\delta\left( s\right) ,$ where $\delta$ is the Dirac
mass weighted at $0,$ then (\ref{1b}) becomes Fourier's law (\ref{a2}%
), and if we let 
\begin{equation*}
	g\left( s\right) =\frac{\kappa}{\tau_{0}}e^{-\frac{s}{\tau_{0}}},\;\tau>0,
\end{equation*}
we recover Cattaneo's law (\ref{d}). Therefore, (\ref{1b}) is a generalized
form of Fourier's and Cattaneo's laws.

In the context of Gurtin-Pipkin theory, Pata and Vuk \cite{Pata} considered the
linear thermoelastic system 
\begin{equation*}
	\left\{ 
	\begin{array}{l}
		u_{tt}(x,t)=u_{xx}(x,t)-\theta _{x}(x,t), \\ 
		\theta _{t}\left( x,t\right) =-u_{tx}\left( x,t\right) -q_{x}\left(
		x,t\right) ,%
	\end{array}%
	\right.
\end{equation*}%
where the heat flux $q$ is modeled by (\ref{1b}). They used the semigroup
theory and established an exponential stability result under some assumptions on $\mu \left( s\right)
=-g^{\prime }\left( s\right) $. Fatori and Mu\~{n}oz Rivera \cite{Fatouri} studied the
system 
\begin{equation*}
	\left\{ 
	\begin{array}{l}
		u_{tt}-au_{xx}+\alpha \theta _{x}=0\;in\;(0;L)\times 
		\mathbb{R}
		_{+} \\ 
		\theta _{t}-k\ast \theta _{xx}+\alpha u_{xt}=0\;in\;(0;L)\times 
		\mathbb{R}
		_{+},%
	\end{array}%
	\right.
\end{equation*}%
where 
\begin{equation*}
	\left( k\ast \theta _{xx}\right) \left( t\right) =\int_{0}^{t}k\left( t-\tau
	\right) \theta _{xx}\left( \tau \right) d\tau.
\end{equation*}%
They obtained an exponential decay result provided that the kernel $k$ is
positive definite and decays exponentially.

Concerning Timoshenko type systems with heat flux given by Gurtin-Pipkin law, Dell'Oro an Pata \cite{Oro} analyzed the following
system 
\begin{equation}
	\left\{ 
	\begin{array}{l}
		\rho_{1}\varphi_{tt}-\kappa\left( \varphi_{x}+\psi\right) _{x}=0, \\ 
		\rho_{2}\psi_{tt}-b\psi_{xx}+\kappa\left( \varphi_{x}+\psi\right)
		+\delta\theta_{x}=0, \\ 
		\rho_{3}\theta_{t}-\displaystyle\frac{1}{\beta}\int_{0}^{+\infty}g\left(
		s\right) \theta_{xx}\left( t-s\right) ds+\delta\psi_{tx}=0.%
	\end{array}
	\right.  \label{c}
\end{equation}
They introduced the stability constant 
\begin{equation*}
	\chi_{g}=\left[ \frac{\rho_{1}}{\rho_{3}\kappa}-\frac{\beta}{g\left(
		0\right) }\right] \left[ \frac{\rho_{1}}{\kappa}-\frac{\rho_{2}}{b}\right] - 
	\frac{\beta}{g\left( 0\right) }\frac{\rho_{1}\delta^{2}}{\rho_{3}\kappa b}
\end{equation*}
and proved that the solution of the system is exponentially stable if and only if $\chi_{g}=0.$

Recently, Fareh \cite{FarehRV} considered the porous thermoelastic system\begin{equation}\label{FA}
	\left\{ 
	\begin{array}{l}
		\rho u_{tt}=\mu u_{xx}+b\varphi_{x}-\beta\theta_{x}, \\ 
		J\varphi_{tt}=\alpha\varphi_{xx}-bu_{x}-\xi\varphi+m\theta-\tau\varphi_{t}, \\ 
		c\theta_{t}=-\displaystyle\int_{-\infty}^{0}k\left(t-
		s\right) \theta_{xx}\left( s\right) ds-\beta u_{tx}-m\varphi_{t},%
	\end{array}
	\right. 
\end{equation}
where the heat conduction is modeled by Gurtin-Pipkin thermal law (\ref{1b}). He proved the exponential stability of the solution of (\ref{FA}) without any restriction on the coefficients of the system.

The system (\ref{A1}) considered in the present paper is a (weak) version of (\ref{FA}), where the coupling porosity-temperature is direct, however, the coupling elasticity-temperature is through the porosity equation. Moreover, system (\ref{A1}) contain only one dissipative mechanism produced by the heat equation of hyperbolic type. 

Here, we introduce a stability number $\chi _{g}$, for which
the solution of (\ref{A1}) is exponentially stable if and only if $\chi
_{g}=0$. 

The system (\ref{A1}) is supplemented by the following initial and boundary
conditions 
\begin{equation}
	\left\{ 
	\begin{array}{c}
		u(x,0)=u_{0}(x),~\phi (x,0)=\phi _{0}(x),~\theta (x,0)=\theta _{0}(x), \\ 
		u_{t}(x,0)=u_{1}(x),~\phi _{t}(x,0)=\phi _{1}(x),%
	\end{array}
	\right.  \label{2}
\end{equation}
\begin{equation}
	u(0,t)=u(\pi ,t)=\phi _{x}(0,t)=\phi _{x}(\pi ,t)=\theta (0,t)=\theta (\pi
	,t)=0.  \label{3}
\end{equation}

To write the problem in the semigroup setting, we introduce the new variables%
\begin{equation*}
	\theta^{t}(x,s):=\theta(x,t-s),\;s\geq0,
\end{equation*}
and 
\begin{equation*}
	\eta(x,s)=\eta^{t}(x,s):=\int_{0}^{s}\theta^{t}(x,\tau)d\tau,\;s\geq0,
\end{equation*}
which denote the past history and the integrated past history of $\theta$ up to $%
t$, respectively.

Clearly $\eta ^{t}\left( x,s\right) $ satisfies the boundary conditions%
\begin{equation*}
	\eta (0,s)=\eta (\pi ,s)=0.
\end{equation*}%
Moreover, we assume that $g\left( \infty \right) =0$ and $\eta (x,0)=%
\underset{s\longrightarrow 0}{\lim }\eta ^{t}(x,s)=0,$ then%
\begin{equation*}
	q=-\int_{-\infty }^{t}g\left( t-s\right) \theta _{x}\left( x,s\right)
	ds=\int_{0}^{+\infty }g^{\prime }\left( s\right) \eta _{x}^{t}\left(
	x,s\right) ds.
\end{equation*}%
Further, we have 
\begin{equation}
	\eta _{t}(x,s)=\theta -\eta _{s}(x,s).  \label{1a}
\end{equation}%
Setting $\kappa \left( s\right) =-g^{\prime }\left( s\right) ,$ system (\ref%
{A1}) becomes 
\begin{equation}
	\left\{ 
	\begin{array}{ll}
		\rho u_{tt}=\mu u_{xx}+b\phi _{x} & {in}~~(0,\pi )\times (0,\infty ), \\ 
		J\phi _{tt}=\alpha \phi _{xx}-bu_{x}-\xi \phi -\beta \theta _{x} & {in}%
		~~(0,\pi )\times (0,\infty ), \\ 
		c\theta _{t}=-\beta \phi _{xt}+\displaystyle\int_{0}^{+\infty }\kappa
		(s)\eta _{xx}^{t}(x,s)ds & {in}~~(0,\pi )\times (0,\infty ), \\ 
		\eta _{t}^{t}=\theta -\eta _{s}^{t} & {in}~~(0,\pi )\times (0,\infty ),%
	\end{array}%
	\right.  \label{1}
\end{equation}

From (\ref{1})$_{2}$ and the boundary conditions (\ref{3}), we have

\begin{equation*}
	J\frac{d^{2}}{dt^{2}}\int_{0}^{\pi }\phi \left( x,t\right) dx+\xi
	\int_{0}^{\pi }\phi \left( x,t\right) dx=0,
\end{equation*}

which gives

\begin{equation*}
	\int_{0}^{\pi }\phi \left( x,t\right) dx=\left( \int_{0}^{\pi }\phi
	_{0}\left( x\right) dx\right) \cos \left( \sqrt{\frac{\xi }{J}}t\right) +%
	\sqrt{\frac{\xi }{J}}\left( \int_{0}^{\pi }\phi _{1}\left( x\right)
	dx\right) \sin \left( \sqrt{\frac{\xi }{J}}t\right) .
\end{equation*}

Consequently, if we set

\begin{align*}
	\overline{\phi }\left( x,t\right) =&\phi \left( x,t\right) -\left(
	\int_{0}^{\pi }\phi _{0}\left( x\right) dx\right) \cos \left( \sqrt{\frac{%
			\xi }{J}}t\right) \\
		&-\sqrt{\frac{\xi }{J}}\left( \int_{0}^{\pi }\phi
	_{1}\left( x\right) dx\right) \sin \left( \sqrt{\frac{\xi }{J}}t\right) ,
\end{align*}

we find

\begin{equation*}
	\int_{0}^{\pi }\overline{\phi }\left( x,t\right) dx=0.
\end{equation*}
Moreover, we check easily that $\left( u,\overline{\phi },\theta ,\eta
\right) $ solves (\ref{1}) subjected to the boundary conditions (\ref{3}),
initial data 
\begin{equation}
	\left\{ 
	\begin{array}{l}
		u(x,0)=u_{0}(x),~\overline{\phi} (x,0)=\phi _{0}(x)-\displaystyle%
		\int_{0}^{\pi}\phi(x)dx,~\theta (x,0)=\theta _{0}(x), \\ 
		u_{t}(x,0)=u_{1}(x),~\overline{\phi} _{t}(x,0)=\phi _{1}(x)-\displaystyle%
		\sqrt{\frac{\xi }{J}} \int_{0}^{\pi }\phi _{1}\left( x\right) dx,~\eta
		^{0}(x,s)=\eta _{0}(x,s),%
	\end{array}
	\right.  \label{A2}
\end{equation}
and more importantly, Poincar\'{e}'s inequality can be applied for $%
\overline{\phi }$. In the sequel we work with $\left( u,\overline{\phi}
,\theta ,\eta \right) $ but we write $\left( u,\phi ,\theta ,\eta \right) $
for convenience.

Regarding the memory kernel, we assume the following set of hypotheses:

\begin{itemize}
	\item[(h1)] $\kappa \in C\left( R^{+}\right) \cap L^{1}\left( R^{+}\right) .$
	
	\item[(h2)] $\kappa \left( s\right) >0,\kappa ^{^{\prime }}\left( s\right)
	\leq 0,\forall s\geq 0.$
	
	\item[(h3)] $\int_{0}^{+\infty }\kappa (s)ds=g\left( 0\right),$
	
	\item[(h4)] there exists $\delta >0$ such that $\kappa ^{\prime }\left(
	s\right) \leq -\delta \kappa \left( s\right) ,\forall s\geq 0.$
\end{itemize}

To the best of our knowledge, the only contribution that examined the behavior of
porous thermoelastic systems with Gurtin-Pipkin law was done by Fareh \cite{FarehRV}. Here, we improves the results of Fareh \cite{FarehRV} in the sense that we consider the direct coupling only between porosity and temperature and omit the strong dissipation given by $-\tau\phi_t$. Moreover, our result improves that of \cite{Oro} obtained for Timoshenko systems and generalizes the results of 
\cite{Fareh5} and \cite{MunozQuintanilla} since (\ref{1b}) subsumes (\ref{a2}
) and (\ref{d}).

The rest of the paper is organized as follow: in Section 2,
we introduce some functional preliminaries. Section 3 is devoted to the
prove of the existence of a unique solution to (\ref{1}), (\ref{3}), (\ref%
{A2}). In Section 4, we state and prove our stability result provided that $%
\chi _{g}=0$. Finally, in Section 5, we prove the lack of the exponential
decay whenever $\gamma _{g}=0$ or $\chi _{g}\neq 0.$

\section{Preliminaries}

First, let $A:D(A)\subset L^{2}\left( 0,\pi\right)\rightarrow L^{2}\left(
0,\pi\right)$ be the operator defined by $Au=-D^{2}u$. For Dirichlet
boundary conditions $D\left( A\right) =H^{2}\cap H_{0}^{1}$ and the operator 
$A$ is a self-adjoint and positive operator. Therefore, it is possible to define
the powers $A^{r}$ of $A$ for $r\in\mathbb{R}$ and the Hilbert space $%
V_{r}=D\left( A^{r/2}\right) $ equipped with the inner product%
\begin{equation*}
	\left\langle u,v\right\rangle _{r}=\left\langle
	A^{r/2}u,A^{r/2}v\right\rangle
\end{equation*}
and let $\left\Vert u\right\Vert _{r}$ be the associated norm. In
particular, $V_{0}=L^{2},$ $V_{-1}=H^{-1},$ $V_{1}=H_{0}^{1}$ and%
\begin{equation*}
	\left\langle A^{1/2}u,A^{1/2}v\right\rangle =\left\langle Du,Dv\right\rangle
	,\;\forall u,v\in H_{0}^{1}.
\end{equation*}
For $r_{1}>r_{2}$ the injection $V_{r_{1}}\hookrightarrow V_{r_{2}}$ is
continuous.

Next, we introduce the weighted Hilbert space%
\begin{equation*}
	\mathcal{V}=L_{\kappa }^{2}((0,+\infty );H_{0}^{1}(0,\pi )),
\end{equation*}%
with inner product 
\begin{equation*}
	\langle \eta ,\zeta \rangle _{\mathcal{V}}=\int_{0}^{+\infty }\kappa
	(s)\langle \eta _{x}(s),\zeta _{x}(s)\rangle ds,
\end{equation*}%
and the norm 
\begin{equation*}
	\Vert \eta \Vert _{\mathcal{V}}^{2}=\int_{0}^{+\infty }\kappa (s)\Vert \eta
	_{x}(s)\Vert ^{2}ds.
\end{equation*}%
At this point, we introduce the energy functional associated to the solution $\left(
u,\phi ,\theta ,\eta \right) $ of (\ref{1}), (\ref{3}), (\ref{A2}), by%
\begin{eqnarray}
	E\left( t\right) &:&=\frac{1}{2}\int_{0}^{\pi }\left[ \rho u_{t}^{2}+J\phi
	_{t}^{2}+\mu u_{x}^{2}+\xi \phi ^{2}+2bu_{x}\phi +\alpha \phi
	_{x}^{2}+c\theta ^{2}\right] dx  \label{e} \\
	&&+\int_{0}^{+\infty }\kappa (s)\int_{0}^{\pi }\eta _{x}^{2}(s)dxds.  \notag
\end{eqnarray}%
The energy space is then 
\begin{equation*}
	\mathcal{H}:=H_{0}^{1}\left( 0,\pi \right) \times L^{2}\left( 0,\pi \right)
	\times H_{\ast }^{1}\left( 0,\pi \right) \times L^{2}\left( 0,\pi \right)
	\times L^{2}\left( 0,\pi \right) \times \mathcal{V}.
\end{equation*}%
It is a Hilbert space with respect to the inner product 
\begin{equation*}
	\begin{array}{cc}
		\langle U,U^{\ast }\rangle = & \rho \displaystyle\int_{0}^{\pi }vv^{\ast
		}dx+\mu \int_{0}^{\pi }u_{x}u_{x}^{\ast }dx+\xi \int_{0}^{\pi }\phi \phi
		^{\ast }dx+\alpha \int_{0}^{\pi }\phi _{x}\phi _{x}^{\ast }dx \\ 
		& +J\displaystyle\int_{0}^{\pi }\psi \psi ^{\ast }dx+b\int_{0}^{\pi }\phi
		u_{x}^{\ast }dx+b\int_{0}^{\pi }u_{x}\phi ^{\ast }dx \\ 
		& +c\displaystyle\int_{0}^{\pi }\theta \theta ^{\ast }dx+\int_{0}^{+\infty
		}\kappa (s)\int_{0}^{\pi }\eta _{x}(s)\eta _{x}^{\ast }(s)dxds,%
	\end{array}%
\end{equation*}%
where, $U=(u,v,\phi ,\psi ,\eta )^{T}$ and $U^{\ast }=(u^{\ast },v^{\ast
},\phi ^{\ast },\psi ^{\ast },\eta ^{\ast })^{T}$. The associated norm is 
\begin{equation}
	\Vert U\Vert _{\mathcal{H}}^{2}=\mu \Vert u_{x}\Vert ^{2}+\xi \Vert \phi
	\Vert ^{2}+2b\langle \phi ,u_{x}\rangle +\rho \Vert v\Vert ^{2}+\alpha \Vert
	\phi _{x}\Vert ^{2}+J\Vert \psi \Vert ^{2}+c\Vert \theta \Vert ^{2}+\Vert
	\eta \Vert _{\mathcal{V}}^{2}.  \label{5}
\end{equation}

\begin{remark}
	Under the hypothesis $\mu \xi >b^{2}$ we have 
	\begin{equation}
		\int_{0}^{\pi }\left( \mu u_{x}^{2}+\xi \phi ^{2}+2bu_{x}\phi \right)
		dx=\left\Vert \sqrt{\mu }u_{x}+\dfrac{b}{\sqrt{\mu }}\phi \right\Vert
		^{2}+\left( \xi -\dfrac{b^{2}}{\mu }\right) \Vert \phi \Vert ^{2}.  \label{6}
	\end{equation}
\end{remark}

\begin{lemma}
	The energy $E(t)$ defined by (\ref{e}), satisfies along the solution $%
	(u,\phi ,\theta ,\eta )$ the estimate 
	\begin{equation}
		\frac{d}{dt}E(t)=\langle T\eta ,\eta \rangle .  \label{e1}
	\end{equation}
\end{lemma}

\begin{proof}
	Taking the $L^{2}-$inner product of the first three equations of (\ref{1})
	by $u_{t},\phi _{t},\theta $ respectively  and the $\mathcal{V-}$inner
	product of the forth equation by $\eta _{t},$ then adding the obtained
	equations and using integration by parts, (\ref{e1}) follows immediately. 
\end{proof}

\section{Well posedness}

In this section we prove that the problem given by (\ref{1}),(\ref{3}) and (%
\ref{A2}) has a unique solution. The prove is based on the theory of
semigroups and Hille-Yoside Theorem.\newline
To rewrite the problem in the semigroup setting, we introduce two new
independent variables $u_{t}=v$ and $\phi _{t}=\psi $, then the system %
\eqref{1} and the initial conditions \eqref{3} can written as follows 
\begin{equation}
	\left\{ 
	\begin{array}{c}
		U_{t} +\mathcal{A}U=0, \\ 
		U\left(0\right) =U_{0},%
	\end{array}%
	\right.  \label{7}
\end{equation}%
where $\mathcal{A}$ is the operator defined on $\mathcal{H}$ by 
\begin{equation}
	\mathcal{A}U=\left( 
	\begin{array}{c}
		-v \\ 
		-\dfrac{\mu}{\rho }u_{xx}-\dfrac{b}{\rho }\phi _{x} \\ 
		-\psi \\ 
		-\dfrac{\alpha }{J}\phi _{xx} +\dfrac{\xi }{J}\phi +\dfrac{b}{J}u_{x}+\dfrac{%
			\beta }{J}\theta \\ 
		\dfrac{\beta }{c}\psi -\dfrac{1}{c}\displaystyle\int_{0}^{+\infty }\kappa
		(s)\eta _{xx}\left( s\right) ds \\ 
		-\theta +\eta _{s}(s)%
	\end{array}%
	\right) ,  \label{8}
\end{equation}%
with domain 
\begin{equation*}
	D\left( \mathcal{A}\right) =\left\{ 
	\begin{array}{c}
		U\in \mathcal{H}:u,\phi \in H^{2}(0,\pi ),v,\theta \in H_{0}^{1}(0,\pi
		),\psi \in H_{\ast }^{1}(0,\pi ), \\ 
		\eta \in H_{\kappa }^{1}((0,+\infty );H_{0}^{1}), \\ 
		\displaystyle\int_{0}^{+\infty }\kappa (s)\eta _{xx}(s)ds\in
		L^{2}(0,\pi),\eta (0)=0.%
	\end{array}%
	\right\}
\end{equation*}

The well-posedness result reads as follow:

\begin{theorem}
	\label{TH1} Suppose that $\kappa$ satisfies the hypothesis (h1)-(h4), then
	for any $U_{0}=\left( u_{0},u_{1},\phi_{0},\phi_{1}\right. $,$\left.
	\theta_{0},\eta_{0}\right) ^{T}\in\mathcal{H}$ the problem (\ref{7}) has a
	unique solution $U\in C\left( \left( 0,+\infty\right) ;\mathcal{H}\right) $.
	Moreover, if $U_{0}\in D\left( \mathcal{A}\right) $, then the solution $U$
	satisfies 
	\begin{equation*}
		U\in C\left( \left( 0,+\infty\right) ;D\left( \mathcal{A}\right) \right)
		\cap C^{1}\left( \left( 0,+\infty\right) ;\mathcal{H}\right) .
	\end{equation*}
\end{theorem}

The proof of Theorem \ref{TH1} is based on the Hille-Yosida Theorem.

\begin{theorem}[Hille-Yosida]
	\label{TH2}\cite{Brezis} Let $\mathcal{A}$ be a maximal monotone operator on
	a Hilbert space $\mathcal{H}$, then, for any $u_0\in D(\mathcal{A})$ there
	exists a unique function 
	\begin{equation*}
		u\in C^1((0,+\infty);\mathcal{H})\cap C((0,+\infty);D(\mathcal{A})),
	\end{equation*}%
	satisfying 
	\begin{equation*}
		\left\{ 
		\begin{array}{l}
			\dfrac{du}{dt}+\mathcal{A}u=0, \\ 
			u(0)=u_0.%
		\end{array}%
		\right.
	\end{equation*}
\end{theorem}

The proof of Theorem \ref{TH1} will be established through several lemmas.

\begin{lemma}
	\label{L1} Let $T$ be the operator defined on $\mathcal{V}$ by 
	\begin{equation*}
		T\eta=-\eta_s,
	\end{equation*}
	with domain 
	\begin{equation*}
		D(T)=\left\{\eta\in\mathcal{V}:\eta_{s}\in\mathcal{V},~\underset{%
			s\rightarrow0}{\lim}\Vert\eta_{x}\Vert=0\right\}.
	\end{equation*}
	Then, for every $\eta\in D(T)$, we have 
	\begin{equation}  \label{FF}
		\langle T\eta,\eta\rangle_{\mathcal{V}}=-\frac{1}{2}\int_{0}^{+\infty
		}\kappa ^{\prime }(s)\left\Vert \eta _{x}(s)\right\Vert ^{2}ds
	\end{equation}
	and 
	\begin{equation}  \label{F}
		\delta\left\Vert \eta \right\Vert _{\mathcal{V}}^{2} \leq-\int_{0}^{+\infty
		}\kappa ^{\prime }(s)\left\Vert \eta _{x}(s)\right\Vert ^{2}ds.
	\end{equation}
\end{lemma}

\begin{proof}
	\begin{equation*}
		\begin{array}{rl}
			2\langle T\eta,\eta\rangle_{\mathcal{V}}&=\displaystyle-\int_{0}^{
				+\infty}\kappa(s)\frac{d}{ds}\Vert\eta_{x}(s)\Vert^2ds\\
			&=-\displaystyle\left.
			\kappa(s)\Vert\eta_{x}(s)\Vert^{2}\right\vert _{0}^{+\infty}
			+\int_{0}^{+\infty
			}\kappa^{\prime}(s)\Vert\eta_{x}(s)\Vert^{2}ds
	\end{array}\end{equation*}
	\begin{equation}\label{A3}
		2\langle T\eta,\eta\rangle_{\mathcal{V}}=\displaystyle\lim_{s\rightarrow0}\kappa(s)\Vert{\eta}_{x}(s)\Vert^{2}-\lim_{s\rightarrow\infty}\kappa(s)\Vert{\eta}_{x}(s)\Vert^{2}+\int_{0}^{+\infty
		}\kappa^{\prime}(s)\Vert\eta_{x}(s)\Vert^{2}ds.
	\end{equation}
	Since $\kappa\left(  s\right)  \Vert\eta_{x}\left(  s\right)  \Vert$ and
	$\kappa\left(  s\right)  \Vert\eta_{xs}\left(  s\right)  \Vert$ are in 
	$L^{1}\left(\mathbb{R}^{+}\right)  $ and $\eta_{x}(0)=0$, then, the first term in the right hand side of (\ref{A3}) worth
	\begin{align*}
		\lim_{s\rightarrow0}\kappa(s)\Vert{\eta}_{x}(s)\Vert^{2}  &  =\lim
		_{s\rightarrow0}\kappa(s)\left\Vert \int_{0}^{s}{\eta}_{xs}(\tau
		)d\tau\right\Vert ^{2},\\
		&  \leq\limsup_{s\rightarrow0}\left(  \int_{0}^{s}\kappa(s)^{1/2}\left\Vert
		{\eta}_{xs}(\tau)\right\Vert d\tau\right)  ^{2}.
	\end{align*}
	The use of Cauchy-Schwarz inequality, leads to%
	\[
	\lim_{s\rightarrow0}\kappa(s)\Vert{\eta}_{x}(s)\Vert^{2}\leq\limsup
	_{s\rightarrow0}s\int_{0}^{s}\kappa(\tau)\left\Vert {\eta}_{xs}(\tau
	)\right\Vert ^{2}d\tau=0.
	\]
	Therefore,%
	\[2\langle T\eta,\eta\rangle_{\mathcal{V}}=-\lim
	_{s\rightarrow\infty}\kappa(s)\Vert\eta_{x}(s)\Vert^{2}+\int_{0}^{+\infty
	}\kappa^{\prime}(s)\Vert\eta_{x}(s)\Vert^{2}ds.
	\]
	The left-hand side of the last equality is bounded, and from (h2) both terms
	of the right-hand side are non-positive, we infer that the above limit exists
	and is finite, and consequently equals zero. Thus%
	\[\langle T\eta,\eta\rangle_{\mathcal{V}}=\frac{1}{2}\int_{0}^{+\infty}\kappa^{\prime}(s)\Vert\eta_{x}(s)\Vert
	^{2}ds,
	\]
	Moreover, the use of the assumption (h4) on $\kappa $ leads to (\ref{F}), which completes the proof.
\end{proof}

\begin{remark}
	According to (\ref{FF}) and (\ref{e1}), the energy $E(t)$, satisfies 
	\begin{equation}  \label{e2}
		\frac{d}{dt}E(t)=\frac{1}{2}\int_{0}^{+\infty}\kappa^{\prime}(s)\Vert%
		\eta_{x}(s)\Vert^2ds\leq 0.
	\end{equation}
\end{remark}

\begin{lemma}
	\label{L2} The operator $\mathcal{A}$ defined by (\ref{8}) is monotone.
\end{lemma}

\begin{proof}
	A direct calculation using integration by parts and boundary conditions gives
	\begin{align*}
		\left\langle \mathcal{A}U,U\right\rangle _{\mathcal{H}} =&-\int_{0}^{\pi }%
		\Big[\left( \mu u_{xx}+b\phi _{x}\right) \Big]vdx-\mu\int_{0}^{\pi
		}v_{x}u_{x}dx-\xi \int_{0}^{\pi }\psi \phi dx-\alpha \int_{0}^{\pi }\psi
		_{x}\phi _{x}dx \\
		& -\int_{0}^{\pi }\Big[\alpha \phi _{xx}-\xi \phi -bu_{x}-\beta \theta \Big]%
		\psi dx-b\int_{0}^{\pi }\psi u_{x}dx-b\int_{0}^{\pi }v_{x}\phi dx \\
		& -\int_{0}^{\pi }\Big[-\beta \psi +\int_{0}^{+\infty }\kappa (s)\eta
		_{xx}\left( s\right) ds\Big]\theta dx\\
		&-\int_{0}^{+\infty }\kappa (s)\int_{0}^{\pi
		}\Big[\theta _{x}-\eta _{sx}(s)\Big]\eta _{x}\left( s\right) dx \\
		 =&\displaystyle\int_{0}^{+\infty }\kappa (s)
		\eta _{x}\eta_{sx}ds\\
		=&\dfrac{1}{2}\displaystyle\int_{0}^{+\infty }\kappa (s)\dfrac{d}{ds}\Vert
		\eta _{x}\left( s\right) \Vert ^{2}ds.
	\end{align*}
	Therefore, the result follows from Lemma \ref{L1}.
\end{proof}

\begin{lemma}
	\label{L3} The operator $\mathcal{A}$ defined by (\ref{8}) is maximal.
\end{lemma}

\begin{proof}
	Let $F=(f^{1},f^{2},f^{3},f^{4},f^{5},f^{6})\in 
	\mathcal{H}$, we will find  $U\in D(\mathcal{A})$ unique, such that $(I-
	\mathcal{A})U=F$, that is the equation becomes the system 
	\begin{align}
		u-v& =f^{1}, \label{9} \\
		\rho v-\mu u_{xx}-b\phi _{x}& =\rho f^{2}, \label{10} \\
		\phi -\psi & =f^{3}, \label{11} \\
		J\psi -\alpha \phi _{xx}+\xi \phi +bu_{x}+\beta \theta & =Jf^{4},\label{12} \\
		c\theta +\beta \psi -\displaystyle\int_{0}^{+\infty }\kappa (s)\eta _{xx}(s)ds&
		=cf^{5}, \label{13}\\
		\eta (s)-\theta +\eta _{s}(s)& =f^{6}.  \label{14}
	\end{align}%
	By solving \eqref{14} subject to the initial condition $\eta (0)=0$, we get 
	\begin{equation}\label{15a}
		\eta (s)=(1-e^{s})\theta +\int_{0}^{s}e^{y-s}f^{6}(y)dy.
	\end{equation}
	Next, substituting $v,\psi$ and $\eta$ from \eqref{9}, \eqref{11} and  \eqref{15a} into (\ref{10}), (\ref{12}) and (\ref{14}), we get
	\begin{equation}
		\left\{ 
		\begin{array}{l}
			au_{xx}+b\phi _{x}-\rho u=-\rho (f^{1}+f^{2}), \\ 
			\alpha \phi _{xx}-bu_{x}+\beta \theta -(J+\xi )\phi =-J\left(
			f^{3}-f^{4}\right)  \\ 
			c_{\mu }\theta _{xx}-c\theta -\beta \phi =-\beta f^{3}-cf^{5}-\displaystyle%
			\int_{0}^{+\infty }\kappa (s)\int_{0}^{s}e^{y-s}f_{xx}^{6}(y)dy,
		\end{array}%
		\right.   \label{15}
	\end{equation}%
	where 
	\begin{equation*}
		c_{\kappa }=\int_{0}^{+\infty }\kappa (s)(1-e^{-s})ds.
	\end{equation*}%
	The last term in the right-hand side of the third equation of \eqref{15}
	belongs to $H^{-1}(0,\pi)$. Indeed, let $\varphi \in H_{0}^{1}(0,\pi)$ such that $\Vert
	\varphi _{x}\Vert \leq 1,$ then 
	\begin{equation*}
		\begin{array}{ll}
			\Big\vert\Big\langle\displaystyle\int_{0}^{+\infty }\kappa (s)\Big(%
			\int_{0}^{s}e^{y-s}f_{xx}^{6}(y)dy\Big)ds,\varphi \Big\rangle\Big\vert= & %
			\Big\vert\Big\langle\displaystyle\int_{0}^{+\infty }\kappa (s)\Big(%
			\int_{0}^{s}e^{y-s}f_{x}^{6}(y)dy\Big)ds,\varphi _{x}\Big\rangle\Big\vert \\ 
			& \leq \displaystyle\int_{0}^{+\infty }\kappa (s)e^{-s}\Big(\int_{0}^{s}e^{y}%
			\Vert f_{x}^{6}(y)\Vert dy\Big)ds \\ 
			& = \displaystyle\int_{0}^{+\infty }e^{y}\Vert f_{x}^{6}(y)\Vert
			\int_{y}^{+\infty}\kappa (s)e^{-s}dsdy \\ 
			& \leq \displaystyle\int_{0}^{+\infty }\kappa (y)e^{y}\Vert f_{x}^{6}(y)\Vert
			\int_{y}^{+\infty}e^{-s}dsdy \\ 
			& \leq \displaystyle\int_{0}^{+\infty }\kappa(y)\Vert f_{x}^{6}(y)\Vert <+\infty.
		\end{array}%
	\end{equation*}%
	
	In view of the above, we multiply the equations (\ref{15})$_{1},$(\ref{15})$_{2}$ and
	(\ref{15})$_{3}$ by $\widetilde{u},\widetilde{\phi}$ and $\widetilde
	{\theta}$ respectively, integrating over $\left(  0,\pi\right)  $ and summing
	up, we obtain the following variational formulation 
	\begin{equation}
		B((u,\phi ,\theta ),(u^{\ast },\phi ^{\ast },\theta ^{\ast }))=L(u^{\ast
		},\phi ^{\ast },\theta ^{\ast })  \label{16}
	\end{equation}	where, 	\begin{equation*}
		\begin{array}{ll}
			B((u,\phi ,\theta ),(u^{\ast },\phi ^{\ast },\theta ^{\ast }))=&\mu\displaystyle%
			\int_{0}^{\pi }u_{x}u_{x}^{\ast }dx-b\int_{0}^{\pi }\phi _{x}u^{\ast
			}dx+\rho \int_{0}^{\pi }uu^{\ast }dx\\
		&+\displaystyle\alpha \int_{0}^{\pi }\phi _{x}\phi
			_{x}^{\ast }dx +b\int_{0}^{\pi }u_{x}\phi ^{\ast }dx-\beta \int_{0}^{\pi
			}\theta \phi ^{\ast }dx\\
		&\displaystyle+(J+\xi )\int_{0}^{\pi }\phi \phi ^{\ast }dx+c_{\kappa
			}\int_{0}^{\pi }\theta _{x}\theta _{x}^{\ast }dx\\
		&\displaystyle+c\int_{0}^{\pi }\theta
			\theta ^{\ast }dx 
			+\beta \int_{0}^{\pi }\phi \theta ^{\ast }dx,%
		\end{array}%
	\end{equation*} 
	is the bilinear form over $\mathcal{W}=H_{0}^{1}\left( 0,\pi \right) \times H^{1}\left( 0,\pi
	\right) \times L^{2}\left( 0,\pi \right) $,
	and  
	\begin{equation*}
		\begin{array}{ll}
			L(u^{\ast },\phi ^{\ast },\theta ^{\ast })=&\rho \displaystyle\int_{0}^{\pi
			}(f^{1}+f^{2})u^{\ast }dx+J\int_{0}^{\pi }\left[ f^{3}+f^{4}\right] \phi
			^{\ast }dx+\int_{0}^{\pi }(\beta f^{3}+cf^{5})\theta ^{\ast }dx \\ 
			&+\displaystyle\int_{0}^{\pi }\theta ^{\ast }\int_{0}^{+\infty }\kappa(s)\Big(%
			\int_{0}^{s}e^{y-s}f_{xx}^{6}(y)dy\Big)dsdx%
		\end{array}%
	\end{equation*} is a linear form.
	Clearly, $B$ and $L$ are continuous. Moreover, a straightforward calculation shows that 
	\begin{equation*}
		\begin{array}{ll}
			B(U,U) & \geq  \displaystyle\dfrac{1}{2}(\mu-\dfrac{b^{2}}{%
				\xi })\int_{0}^{\pi }u_{x}^{2}dx+c\int_{0}^{\pi }\theta
			^{2}dx+J\int_{0}^{\pi }\phi ^{2}dx+\alpha \int_{0}^{\pi }\phi _{x}^{2}dx, \\
			& \geq C\Vert U\Vert_{\mathcal{W}} ^{2},%
		\end{array}%
	\end{equation*}
	Thus, $B$ is coercive.
	Consequently, Lax-Milgram theorem guarantees the
	existence of a unique $(u,\phi ,\theta )\in \mathcal{W}$ satisfying %
	\eqref{16}.
	
	Next, we take $(u^{\ast },\phi ^{\ast },\theta ^{\ast })=(u^{\ast
	},0,0)$ in \eqref{16} to get 
	\begin{equation*}
		\displaystyle \mu\int_{0}^{\pi }u_{x}u_{x}^{\ast }dx=\int_{0}^{\pi }(b\phi
		_{x}-\rho u+\rho \left( f^{1}+f^{2}\right) )u^{\ast }dx.
	\end{equation*}%
	Standard arguments of elliptic equations infer tha that $u\in H^2(0,\pi)$ and
	\begin{equation*}
		u_{xx}=-\frac{1}{\mu}\left(b\phi_{x}-\rho u+\rho(f^1 +f^2)\right).
	\end{equation*} Therefore, 
	\begin{equation*}
		u\in H^{2}\left( 0,\pi \right) \cap H_{0}^{1}\left( 0,\pi \right).
	\end{equation*}%
	Similarly, by taking $(u^{\ast },\phi ^{\ast },\theta ^{\ast })=(0,\phi ^{\ast },0)$, 
	we get \begin{equation}\label{T1}
		\alpha \int_{0}^{\pi }\phi _{x}\phi
		_{x}^{\ast }dx=\int_{0}^{\pi }\left(-bu_{x} +\beta \theta -(J+\xi )\phi+J\left( f^{3}+f^{4}\right) \right)\phi
		^{\ast }dx,~\forall \phi^{\star}\in H_{\star}^1(0,\pi).
	\end{equation}
	Here, we are not able to apply elliptic arguments. To do so, we proceed as follows:\\
	Let $\psi\in H^1_0(0,\pi)$ and set $\phi^{\star}=\psi-\int_{0}^{\pi}\psi(x)dx$. Clearly $\phi^{\star}\in H_{\star}^1(0,\pi)$. Plugging, $\phi^{\star}$ in (\ref{T1}) taking into account that $-bu_{x} +\beta \theta -(J+\xi )\phi+J\left( f^{3}+f^{4}\right) \in L_{\star}^2(0,\pi)$, we obtain
	\begin{equation*}
		\alpha \int_{0}^{\pi }\phi _{x}\psi_{x}dx=\int_{0}^{\pi }\left(-bu_{x} +\beta \theta -(J+\xi )\phi+J\left( f^{3}+f^{4}\right) \right)\psi dx,~\forall \psi\in H_{0}^1(0,\pi).
	\end{equation*}
	Thus,
	\begin{equation*}
		\phi \in H^{2}\left( 0,\pi \right) \cap H_{\ast }^{1}\left( 0,\pi \right),
	\end{equation*}
	with \begin{equation*}
		\phi_{xx}=-\frac{1}{\alpha}(-bu_{x} +\beta \theta -(J+\xi )\phi+J\left( f^{3}+f^{4}\right)).
	\end{equation*}
	Finally,  we take $(u^{\ast },\phi ^{\ast },\theta ^{\ast })=(0,0,\theta
	^{\ast })$ to get $\theta \in H_{0}^{1}\left( 0,\pi \right) .$\\
	Back to \eqref{9},\eqref{11} and \eqref{15a} taking into account the last results, we
	get 
	\begin{equation*}
		v\in H_{0}^{1}\left( 0,\pi \right) ,~~\psi \in H_{\ast
		}^{1}\left( 0,\pi \right) ~~\mbox{and}~~\eta \in
		H_{\kappa}^{1}((0,+\infty );H_{0}^{1}).
	\end{equation*}%
	Furthermore, (\ref{13}) yields
	\[
	\int_{0}^{+\infty}\kappa\left(  s\right)  \eta_{xx}^{t}\left(  s\right)  ds\in
	L^{2}(0,\pi).
	\]
	
	Hence, the solution $U$ belongs to $D(\mathcal{A})$ and then $1$ is in the resolvent set of $\mathcal{A}$, which completes the proof of the Lemma.
	
\end{proof}

\begin{proof}[Proof. (of Theorem \ref{TH1})]
	From Lammas \ref{L2}, \ref{L3} and by Hille-Yoside Theorem \ref{TH2}, the problem (\ref{7}) has a unique solution.
\end{proof} 

\begin{remark}
	The operator maximal and monotone $\mathcal{A}$ generates a C$_{0}-$%
	semigroup of contractions $S(t)=e^{-\mathcal{A}t}$, and the solution of (\ref%
	{7}) is given by 
	\begin{equation*}
		U(x,t)=S(t)U_{0}(x),~\forall t\geq 0.
	\end{equation*}%
	Moreover, if $U_{0}\in \mathcal{H}$, then the solution $U$ is a weak
	solution and satisfies 
	\begin{equation*}
		U\in C((0,+\infty );\mathcal{H}).
	\end{equation*}
\end{remark}

\section{Exponential stability}

In this section we state and prove the stability result of our problem. 
\newline
First, define 
\begin{equation}
	\gamma _{g}=c\mu -\rho g\left( 0\right)
\end{equation}
and for $\gamma _{g}\neq 0,$ we introduce the stability number 
\begin{equation}
	\chi _{g}=\frac{\rho }{\mu }-\frac{J}{\alpha }+\frac{\rho \beta ^{2}}{\alpha
		\gamma _{g}}.
\end{equation}%
The main result reads as follow:

\begin{theorem}
	\label{TH3} Let $(u,\phi ,\theta ,\eta )$ be the solution of (\ref{1})
	subjected to the initial and boundary conditions (\ref{3}),(\ref{A2})
	respectively. Assume that $\gamma _{g}\neq 0$ and $\chi _{g}=0,$ then, the
	energy $E\left( t\right) $ associated to the solution $(u,\phi ,\theta ,\eta
	)$ is exponentially stable, that there exist two positive constants $\sigma
	,\omega $ such that%
	\begin{equation*}
		E\left( t\right) \leq \sigma e^{-\omega t},\;\forall t\geq 0.
	\end{equation*}
\end{theorem}

The proof of Theorem \ref{TH3} will be done by the multipliers method and
will be established through several lemmas.

\begin{lemma}
	Let $(u,\phi ,\theta ,\eta )$ be a solution of (\ref{1}), then the
	functional 
	\begin{equation*}
		F_{1}(t)=-\frac{2c}{g(0)}\int_{0}^{+\infty }\kappa (s)\left\langle \theta
		(t),\eta ^{t}(s)\right\rangle ds
	\end{equation*}%
	satisfies, for any $\varepsilon _{1}>0$, the estimate 
	\begin{equation}
		\frac{d}{dt}F_{1}(t)\leq -c\left\Vert \theta \right\Vert ^{2}-\left\Vert
		\eta \right\Vert ^{2}+\varepsilon _{1}\left\Vert \phi _{t}\right\Vert
		^{2}-M\left( 1+\frac{1}{\varepsilon _{1}}\right) \int_{0}^{+\infty }\kappa
		^{\prime }\left( s\right) \Vert \eta _{x}(s)\Vert ^{2}ds,  \label{25}
	\end{equation}%
	where $M$ is a positive constant independent of $\varepsilon _{1}.$
\end{lemma}

\begin{proof}
	Differentiating $F_{1}(t)$, then, using (\ref{1})$_{3}$, (\ref{1})$_{4}$ and
	integration by parts, we infer 
	\begin{align}
		\frac{d}{dt}F_{1}(t) =&-\frac{2c}{g(0)}\int_{0}^{+\infty }\kappa
		(s)\left\langle \theta _{t}(t),\eta (s)\right\rangle ds-\frac{2c}{g(0)}%
		\int_{0}^{+\infty }\kappa (s)\left\langle \theta (t),\eta
		_{t}(s)\right\rangle ds,  \notag \\
		=&-\frac{2}{g(0)}\int_{0}^{+\infty }\kappa (s)\left\langle -\beta \phi
		_{xt}+\int_{0}^{+\infty }\kappa (s)\eta _{xx}(x,s)ds,\eta (s)\right\rangle ds\notag\\
		&-\frac{2c}{g(0)}\int_{0}^{+\infty }\kappa (s)\left\langle \theta ,\theta -\eta
		_{s}(s)\right\rangle ds,  \notag \\
		=&-\frac{2\beta }{g(0)}\int_{0}^{+\infty }\kappa (s)\left\langle \phi
		_{t},\eta _{x}(s)\right\rangle ds+\frac{2}{g(0)}\left\Vert \int_{0}^{+\infty
		}\kappa (s)\eta _{x}(s)ds\right\Vert ^{2} \notag  \\
		& -2c\left\Vert \theta \right\Vert ^{2}+\frac{2c}{g(0)}\int_{0}^{+\infty
		}\kappa (s)\left\langle \theta ,\eta _{s}(s)\right\rangle ds. \label{20}
	\end{align}%
	First, integrating by parts with respect to $s$, using Cauchy Schwarz,
	Young's and Poincar\'{e}'s inequalities we get 
	\begin{align}
		\frac{2c}{g(0)}\int_{0}^{+\infty }\kappa (s)\left\langle \theta ,\eta
		_{s}(s)\right\rangle ds& =-\frac{2c}{g(0)}\left\langle \theta
		,\int_{0}^{+\infty }\kappa ^{\prime }(s)\eta (s)ds\right\rangle   \notag \\
		& \leq c\left\Vert \theta \right\Vert ^{2}-M\int_{0}^{+\infty }\kappa
		^{\prime }\Vert \eta _{x}(s)\Vert ^{2}ds.  \label{21}
	\end{align}%
	Next, Young's inequality and (\ref{F}) yield 
	\begin{align}
		-\frac{2\beta }{g(0)}\int_{0}^{+\infty }\kappa (s)\left\langle \phi _{t},\eta
		_{x}(s)\right\rangle ds&=-\frac{2\beta }{g(0)}\left\langle \phi
		_{t},\int_{0}^{+\infty }\kappa (s)\eta _{x}(s)ds\right\rangle\notag\\ 
		&\leq
		\varepsilon _{1}\left\Vert \phi _{t}\right\Vert ^{2}-\frac{M}{\varepsilon
			_{1}}\int_{0}^{+\infty }\kappa ^{\prime }(s)\left\Vert \eta
		_{x}(s)\right\Vert ^{2}ds.  \label{22}
	\end{align}%
	Cauchy-Schwarz' inequality and (\ref{F}) give%
	\begin{eqnarray}
		\frac{2}{g(0)}\left\Vert \displaystyle\int_{0}^{+\infty }\kappa (s)\eta
		_{x}(s)ds\right\Vert ^{2} &\leq\displaystyle 2\left\Vert \eta \right\Vert _{\mathcal{V}%
		}^{2}\leq -\left\Vert \eta \right\Vert _{\mathcal{V}}^{2}-\frac{3}{\delta }%
		\int_{0}^{+\infty }\kappa ^{\prime }(s)\left\Vert \eta _{x}(s)\right\Vert
		^{2}ds,  \notag \\
		&\leq -\left\Vert \eta \right\Vert _{\mathcal{V}}^{2}-M\displaystyle\int_{0}^{+\infty
		}\kappa ^{\prime }(s)\left\Vert \eta _{x}(s)\right\Vert ^{2}ds.  \label{23}
	\end{eqnarray}%
	Substituting (\ref{21})-(\ref{23}) into (\ref{20}), (\ref{25}) follows
	immediately. 
\end{proof}

\begin{lemma}
	Let $(u,\phi ,\theta ,\eta )$ be the solution of (\ref{1}), the functional 
	\begin{equation*}
		F_{2}(t)=-\frac{cJ}{\beta }\left\langle \theta ,\int_{0}^{x}\phi
		_{t}(y)dy\right\rangle
	\end{equation*}%
	satisfies for any \ $\varepsilon _{2}>0,$ the estimate%
	\begin{align}
		\frac{d}{dt}F_{2}(t)& =-\frac{J}{2}\left\Vert \phi _{t}\right\Vert
		^{2}-M\int_{0}^{+\infty }\kappa ^{\prime }(s)\left\Vert \eta
		_{x}(s)\right\Vert ^{2}ds+M\left( 1+\frac{1}{\varepsilon _{2}}\right)
		\left\Vert \theta \right\Vert ^{2}  \notag \\
		& +\varepsilon _{2}\left\Vert \sqrt{a}u_{x}+\frac{b}{\sqrt{a}}\phi
		\right\Vert ^{2}+\varepsilon _{2}\left\Vert \phi \right\Vert
		^{2}+\varepsilon _{2}\left\Vert \phi _{x}\right\Vert ^{2},  \label{30}
	\end{align}%
	where, $M$ is a positive constant independent of $\varepsilon _{2}.$
\end{lemma}

\begin{proof}%
	Differentiating $F_{2}(t)$, then using (\ref{1})$_{2}$, (\ref{1})$_{3}$ and
	integration by parts, we get 
	\begin{align*}
		\frac{d}{dt}F_{2}(t) =&-\frac{cJ}{\beta }\left\langle \theta
		_{t},\int_{0}^{x}\phi _{t}(y)dy\right\rangle -\frac{cJ}{\beta }\left\langle
		\theta ,\int_{0}^{x}\phi _{tt}(y)dy\right\rangle  \\
		=&-\frac{J}{\beta }\left\langle -\beta \phi _{xt}+\int_{0}^{+\infty }\kappa
		(s)\eta _{xx}^{t}(s)ds,\int_{0}^{x}\phi _{t}(y)dy\right\rangle\notag\\
		& -\frac{c}{%
			\beta }\left\langle \theta ,\int_{0}^{x}\left( \alpha \phi _{xx}-bu_{x}-\xi
		\phi -\beta \theta _{x}\right) dy\right\rangle  \\
		=&-J\left\Vert \phi _{t}\right\Vert ^{2}+\frac{J}{\beta }\int_{0}^{+\infty
		}\kappa (s)\left\langle \eta _{x}^{t}(s),\phi _{t}\right\rangle ds-\frac{%
			c\alpha }{\beta }\left\langle \theta ,\phi _{x}\right\rangle\notag\\
		& +c\left\Vert
		\theta \right\Vert ^{2}+\frac{c}{\beta }\left\langle \theta
		,\int_{0}^{x}\left( bu_{x}+\xi \phi \right) dy\right\rangle .
	\end{align*}%
	Notice that 
	\begin{align*}
		\frac{c}{\beta }\left\langle \theta ,\int_{0}^{x}\left( bu_{x}+\xi \phi
		\right) \left( y\right) dy\right\rangle & =\frac{cb}{\beta \sqrt{\mu }}%
		\left\langle \theta ,\int_{0}^{x}\left( \sqrt{\mu }u_{x}+\frac{b}{\sqrt{\mu }%
		}\phi \right) \left( y\right) dy\right\rangle  \\
		& +\frac{c}{\beta }\left( \xi -\frac{b^{2}}{\mu }\right) \left\langle \theta
		,\int_{0}^{x}\phi (y)dy\right\rangle ,
	\end{align*}%
	we infer that%
	\begin{align}
		\frac{d}{dt}F_{2}(t)& =-J\left\Vert \phi _{t}\right\Vert ^{2}+\frac{J}{\beta 
		}\int_{0}^{+\infty }\kappa (s)\left\langle \eta _{x}^{t}(s),\phi
		_{t}\right\rangle ds   -\frac{c\alpha }{\beta }\left\langle
		\theta ,\phi _{x}\right\rangle +c\left\Vert \theta \right\Vert ^{2}\label{24} \notag
		\\
		& +\frac{c}{\beta }\left( \xi -\frac{b^{2}}{\mu }\right) \left\langle \theta
		,\int_{0}^{x}\phi (y)dy\right\rangle +\frac{cb}{\beta \sqrt{\mu }}\left\langle \theta
		,\int_{0}^{x}\left( \sqrt{\mu }u_{x}+\frac{b}{\sqrt{\mu }}\phi \right)
		\left( y\right) dy\right\rangle.  	\end{align}%
	Young's and Cauchy Schwarz inequalities and (\ref{F}), give%
	\begin{align}
		\frac{J}{\beta }\int_{0}^{+\infty }\kappa (s)\left\langle \eta _{x}(s),\phi
		_{t}\right\rangle ds& \leq \frac{J}{\beta }\left\Vert \phi _{t}\right\Vert
		\int_{0}^{+\infty }\kappa (s)\left\Vert \eta _{x}\right\Vert ds  \notag \\
		& \leq \frac{J}{2}\left\Vert \phi _{t}\right\Vert ^{2}+M\left\Vert \eta
		\right\Vert _{\mathcal{V}}^{2}  \notag \\
		& \leq \frac{J}{2}\left\Vert \phi _{t}\right\Vert ^{2}-M\int_{0}^{+\infty
		}\kappa ^{\prime }\left( s\right) \Vert \eta _{x}(s)\Vert ^{2}ds.  \label{26}
	\end{align}%
	Similarly, for any $\varepsilon _{2}>0$ we have%
	\begin{equation}
		\frac{cb}{\beta \sqrt{\mu }}\left\langle \theta ,\int_{0}^{x}\left( \sqrt{%
			\mu }u_{x}+\frac{b}{\sqrt{\mu }}\phi \right) \left( y\right) dy\right\rangle
		\leq \frac{M}{\varepsilon _{2}}\left\Vert \theta \right\Vert
		^{2}+\varepsilon _{2}\left\Vert \sqrt{\mu }u_{x}+\frac{b}{\sqrt{\mu }}\phi
		\right\Vert ^{2},  \label{27}
	\end{equation}
	\begin{equation}
		\frac{c}{\beta }\left( \xi -\frac{b^{2}}{a}\right) \left\langle \theta
		,\int_{0}^{x}\phi (y)dy\right\rangle \leq \frac{M}{\varepsilon _{2}}%
		\left\Vert \theta \right\Vert ^{2}+\varepsilon _{2}\left\Vert \phi
		\right\Vert ^{2}  \label{28}
	\end{equation}%
	and 
	\begin{equation}
		c\alpha \left\langle \theta ,\phi _{x}\right\rangle \leq \frac{M}{%
			\varepsilon _{2}}\left\Vert \theta \right\Vert ^{2}+\varepsilon
		_{2}\left\Vert \phi _{x}\right\Vert ^{2}.  \label{29}
	\end{equation}%
	The  substitution of \ (\ref{26})-(\ref{29}) into (\ref{24}), yields (\ref%
	{30}). \end{proof}

\begin{lemma}
	Let%
	\begin{equation*}
		I_{1}=\frac{\alpha \rho }{b}\left\langle \phi _{x},u_{t}\right\rangle +\frac{%
			J\mu }{b}\left\langle \phi _{t},u_{x}\right\rangle +2J\left\langle \phi
		,\phi _{t}\right\rangle -\dfrac{\rho \xi }{b}\left\langle
		u_{t},\int_{0}^{x}\phi (y)dy\right\rangle ,
	\end{equation*}%
	\begin{equation*}
		I_{2}(t)=-\frac{1}{\sqrt{\mu }}\int_{0}^{+\infty }\kappa (s)\left\langle \eta
		_{x}(s),\sqrt{\mu }u_{x}+\frac{b}{\sqrt{\mu }}\phi \right\rangle ds,
	\end{equation*}%
	\begin{equation*}
		I_{3}\left( t\right) =-\frac{\beta \rho }{b}\left\langle \theta
		,u_{t}\right\rangle ,
	\end{equation*}%
	and 
	\begin{equation}
		F_{3}\left( t\right) =I_{1}(t)+\frac{\rho \beta \mu }{b\gamma _{g}}I_{2}(t)+%
		\frac{\mu c}{\gamma _{g}}I_{3}\left( t\right) .  \label{32a}
	\end{equation}
	
	Suppose that $\chi_g=0$, then for all $\varepsilon>0$, the functional $F_3$
	satisfies,
	
	\begin{align}
		\frac{d}{dt}F_{3}\left( t\right) \leq & -\frac{1}{2}\left\Vert \sqrt{\mu }%
		u_{x}+\frac{b}{\sqrt{\mu }}\phi \right\Vert ^{2}-\left( \xi -\frac{b^{2}}{a}%
		\right) \left\Vert \phi \right\Vert ^{2}-\frac{\alpha }{2}\left\Vert \phi
		_{x}\right\Vert ^{2} +M\left( 1+\dfrac{1}{\varepsilon _{3}}\right) \left\Vert \phi
		_{t}\right\Vert ^{2} \notag \\
		& +\varepsilon _{3}\left\Vert u_{t}\right\Vert
		^{2}+M\left\Vert \theta \right\Vert ^{2}-M\left( 1+\varepsilon _{3}\right)
		\int_{0}^{+\infty }\kappa ^{\prime }(s)\Vert \eta _{x}\Vert ^{2}ds.
		\label{35}
	\end{align}
	
	such that $\chi =\frac{\rho }{\mu }-\frac{J}{\alpha }+\frac{\rho \beta ^{2}}{%
		\alpha \gamma _{g}}$ \ , $\gamma _{g}=c\mu -g\left( 0\right) \rho \neq 0$
	and \ 
\end{lemma}

\begin{proof}
	Direct differentiation, using (\ref{1})$_{1}$ and integration by parts,
	yields 
	\begin{align}
		\dfrac{d}{dt}\frac{\alpha \rho }{b}\left\langle \phi _{x},u_{t}\right\rangle
		& =\frac{\alpha \rho }{b}\left\langle \phi _{tx},u_{t}\right\rangle +\frac{%
			\alpha \rho }{b}\left\langle \phi _{x},u_{tt}\right\rangle  \notag \\
		& =\frac{\alpha \rho }{b}\left\langle \phi _{tx},u_{t}\right\rangle +\frac{%
			\alpha }{b}\left\langle \phi _{x},\mu u_{xx}+b\phi _{x}\right\rangle  \notag
		\\
		& =\frac{\alpha \rho }{b}\left\langle \phi _{tx},u_{t}\right\rangle +\frac{%
			\mu \alpha }{b}\left\langle \phi _{x},u_{xx}\right\rangle +\alpha \left\Vert
		\phi _{x}\right\Vert ^{2},  \label{25a}
	\end{align}%
	Similarly, using (\ref{1})$_{2}$ we have%
	\begin{align}
		\dfrac{d}{dt}\frac{J\mu }{b}\left\langle u_{x},\phi _{t}\right\rangle & =%
		\frac{J\mu }{b}\left\langle u_{tx},\phi _{t}\right\rangle +\frac{\mu }{b}%
		\left\langle u_{x},J\phi _{tt}\right\rangle  \notag \\
		& =\frac{\mu }{b}\left\langle u_{x},\alpha \phi _{xx}-bu_{x}-\xi \phi -\beta
		\theta _{x}\right\rangle +\frac{J\mu }{b}\left\langle u_{tx},\phi
		_{t}\right\rangle  \notag \\
		& =-\frac{\mu \alpha }{b}\left\langle u_{x},\phi _{xx}\right\rangle -\mu
		\left\Vert u_{x}\right\Vert ^{2}-\frac{\mu \xi }{b}\left\langle u_{x},\phi
		\right\rangle -\frac{\mu \beta }{b}\left\langle u_{x},\theta
		_{x}\right\rangle +\frac{J\mu }{b}\left\langle u_{tx},\phi _{t}\right\rangle
		\label{26a}
	\end{align}%
	and%
	\begin{align}
		2\dfrac{d}{dt}J\left\langle \phi ,\phi _{t}\right\rangle & =2J\left\Vert
		\phi _{t}\right\Vert ^{2}+2J\left\langle \phi ,\phi _{tt}\right\rangle 
		\notag \\
		& =2J\left\Vert \phi _{t}\right\Vert ^{2}+2\left\langle \phi ,\alpha \phi
		_{xx}-bu_{x}-\xi \phi -\beta \theta _{x}\right\rangle  \notag \\
		& =2J\left\Vert \phi _{t}\right\Vert ^{2}-2\alpha \left\Vert \phi
		_{x}\right\Vert ^{2}-2\xi \left\Vert \phi \right\Vert ^{2}-2b\left\langle
		\phi ,u_{x}\right\rangle -2\beta \left\langle \phi ,\theta _{x}\right\rangle
		.  \label{27a}
	\end{align}%
	By exploiting (\ref{1})$_{1}$ we infer that
	\begin{align}
		-\dfrac{d}{dt}\dfrac{\rho \xi }{b}\left\langle u_{t},\int_{0}^{x}\phi
		(y)dy\right\rangle & =-\dfrac{\rho \xi }{b}\left\langle
		u_{tt},\int_{0}^{x}\phi (y)dy\right\rangle -\dfrac{\rho \xi }{b}\left\langle
		u_{t},\int_{0}^{x}\phi _{t}(y)dy\right\rangle  \notag \\
		& =-\dfrac{\xi }{b}\left\langle \mu u_{xx}+b\phi _{x},\int_{0}^{x}\phi
		(y)dy\right\rangle -\dfrac{\rho \xi }{b}\left\langle u_{t},\int_{0}^{x}\phi
		_{t}(y)dy\right\rangle  \notag \\
		& =\dfrac{\mu \xi }{b}\left\langle u_{x},\phi \right\rangle +\xi \left\Vert
		\phi \right\Vert ^{2}-\dfrac{\rho \xi }{b}\left\langle
		u_{t},\int_{0}^{x}\phi _{t}(y)dy\right\rangle .  \label{28a}
	\end{align}%
	The addition of (\ref{25a})-(\ref{28a}) yields 
	\begin{align*}
		\frac{d}{dt}I_{1}\left( t\right) & =-\mu \left\Vert u_{x}\right\Vert
		^{2}-\xi \left\Vert \phi \right\Vert ^{2}-2b\left\langle \phi
		,u_{x}\right\rangle -\alpha \left\Vert \phi _{x}\right\Vert
		^{2}+2J\left\Vert \phi _{t}\right\Vert ^{2} \\
		& -2\beta \left\langle \phi ,\theta _{x}\right\rangle -\frac{\mu \beta }{b}%
		\left\langle \theta _{x},u_{x}\right\rangle -\dfrac{\rho \xi }{b}%
		\left\langle u_{t},\int_{0}^{x}\phi _{t}(y)dy\right\rangle +\left( \frac{%
			J\mu }{b}-\frac{\alpha \rho }{b}\right) \left\langle u_{tx},\phi
		_{t}\right\rangle .
	\end{align*}%
	Therefore, (\ref{6}) leads to%
	\begin{align}
		\frac{d}{dt}I_{1}\left( t\right) & =-\left\Vert \sqrt{\mu }u_{x}+\frac{b}{%
			\sqrt{\mu }}\phi \right\Vert ^{2}-\left( \xi -\frac{b^{2}}{\mu }\right)
		\left\Vert \phi \right\Vert ^{2}-\alpha \left\Vert \phi _{x}\right\Vert
		^{2}+2J\left\Vert \phi _{t}\right\Vert ^{2}  \notag \\
		& -2\beta \left\langle \phi ,\theta _{x}\right\rangle -\frac{\mu \beta }{b}%
		\left\langle \theta _{x},u_{x}\right\rangle -\dfrac{\rho \xi }{b}%
		\left\langle u_{t},\int_{0}^{x}\phi _{t}(y)dy\right\rangle +\left( \frac{Ja}{%
			b}-\frac{\alpha \rho }{b}\right) \left\langle u_{tx},\phi _{t}\right\rangle .
		\label{29a}
	\end{align}%
	The differentiation of $I_{2}\left( t\right) $ and integration by parts
	yield 
	\begin{align*}
		\frac{d}{dt}I_{2}(t) =&-\frac{1}{\sqrt{\mu }}\int_{0}^{+\infty }\kappa
		(s)\left\langle \eta _{tx}(s),\sqrt{\mu }u_{x}+\frac{b}{\sqrt{\mu }}\phi
		\right\rangle ds\\
		&-\displaystyle\frac{1}{\sqrt{\mu }}\int_{0}^{+\infty }\kappa
		(s)\left\langle \eta _{x}(s),\sqrt{\mu }u_{tx}+\frac{b}{\sqrt{\mu }}\phi
		_{t}\right\rangle ds ,\\
		=&\frac{1}{\sqrt{\mu }}\int_{0}^{+\infty }\kappa (s)\left\langle \eta
		_{t}(s),\sqrt{\mu }u_{xx}+\frac{b}{\sqrt{\mu }}\phi _{x}\right\rangle ds \\
		& +\frac{1}{\sqrt{\mu }}\int_{0}^{+\infty }\kappa (s)\left\langle \eta
		_{xx}(s),\sqrt{\mu }u_{t}\right\rangle ds-\frac{1}{\sqrt{\mu }}%
		\int_{0}^{+\infty }\kappa (s)\left\langle \eta _{x}(s),\frac{b}{\sqrt{\mu }}%
		\phi _{t}\right\rangle ds
	\end{align*}%
	Now, using (\ref{1})$_{3}$, (\ref{1})$_{4}$ and integration by parts, we get%
	\begin{align*}
		\frac{d}{dt}I_{2}\left( t\right) =&\frac{g\left( 0\right) }{\sqrt{\mu }}%
		\left\langle \theta ,\sqrt{\mu }u_{xx}+\frac{b}{\sqrt{\mu }}\phi
		_{x}\right\rangle -\frac{1}{\sqrt{\mu }}\int_{0}^{\infty }\kappa
		(s)\left\langle \eta _{s},\sqrt{\mu }u_{xx}+\frac{b}{\sqrt{\mu }}\phi
		_{x}\right\rangle ds \\
		&+\left\langle c\theta _{t}+\beta \phi _{xt},u_{t}\right\rangle -\frac{b}{%
			\mu }\int_{0}^{+\infty }\kappa (s)\left\langle \eta _{x}(s),\phi
		_{t}\right\rangle ds.
	\end{align*}%
	\begin{align*}
		=&-g(0)\left\langle \theta _{x},u_{x}\right\rangle +\frac{b}{\mu }%
		g(0)\left\langle \theta ,\phi _{x}\right\rangle -\frac{1}{\sqrt{\mu }}%
		\int_{0}^{+\infty }\kappa (s)\left\langle \eta _{s},\sqrt{\mu }u_{xx}+\frac{b%
		}{\sqrt{\mu }}\phi _{x}\right\rangle ds \\
		&+c\left\langle \theta _{t},u_{t}\right\rangle ds-\beta \left\langle \phi
		_{t},u_{xt}\right\rangle -\frac{b}{\mu }\int_{0}^{+\infty }\kappa
		(s)\left\langle \eta _{x}(s),\phi _{t}\right\rangle ds.
	\end{align*}%
	The integration by parts gives%
	\begin{align*}
		-\frac{1}{\sqrt{\mu }}\int_{0}^{+\infty }\kappa (s)\left\langle \eta _{s},%
		\sqrt{\mu }u_{xx}+\frac{b}{\sqrt{\mu }}\phi _{x}\right\rangle ds =\frac{1}{%
			\sqrt{\mu }}\int_{0}^{+\infty }\kappa ^{\prime }(s)\left\langle \eta ,\sqrt{%
			\mu }u_{xx}+\frac{b}{\sqrt{\mu }}\phi _{x}\right\rangle ds
	\end{align*}
\begin{align*}
		 =-\frac{1}{\sqrt{\mu }}\int_{0}^{+\infty }\kappa ^{\prime }(s)\left\langle
		\eta _{x},\sqrt{\mu }u_{x}+\frac{b}{\sqrt{\mu }}\phi \right\rangle ds.
	\end{align*}%
	Thus, 
	\begin{align}
		\frac{d}{dt}I_{2}(t)& =-g(0)\left\langle \theta _{x},u_{x}\right\rangle +%
		\frac{b}{\mu }g(0)\left\langle \theta ,\phi _{x}\right\rangle -\frac{1}{%
			\sqrt{\mu }}\int_{0}^{+\infty }\kappa ^{\prime }(s)\left\langle \eta _{x},%
		\sqrt{\mu }u_{x}+\frac{b}{\sqrt{\mu }}\phi \right\rangle ds  \notag \\
		& +c\left\langle \theta _{t},u_{t}\right\rangle +\beta \left\langle \phi
		_{xt},u_{t}\right\rangle -\frac{b}{\mu }\int_{0}^{+\infty }\kappa
		(s)\left\langle \eta _{x}(s),\phi _{t}\right\rangle ds.  \label{30a}
	\end{align}%
	Finally, using (\ref{1})$_{1}$ we get 
	\begin{align}
		\frac{d}{dt}I_{3}\left( t\right) & =-\frac{\beta \rho }{b}\left\langle
		\theta _{t},u_{t}\right\rangle -\frac{\beta \rho }{b}\left\langle \theta
		,u_{tt}\right\rangle  \notag \\
		& =-\frac{\beta \rho }{b}\left\langle \theta _{t},u_{t}\right\rangle -\frac{%
			\beta }{b}\left\langle \theta ,\mu u_{xx}+b\phi _{x}\right\rangle  \notag \\
		& =-\frac{\beta \rho }{b}\left\langle \theta _{t},u_{t}\right\rangle +\frac{%
			\beta \mu }{b}\left\langle \theta _{x},u_{x}\right\rangle -\beta
		\left\langle \theta ,\phi _{x}\right\rangle .  \label{31}
	\end{align}%
	Plugging (\ref{29a})-(\ref{31}) into (\ref{32a}) we obtain%
	\begin{align*}
		\frac{d}{dt}F_{3}\left( t\right) & =-\left\Vert \sqrt{\mu }u_{x}+\frac{b}{%
			\sqrt{\mu }}\phi \right\Vert ^{2}-\left( \xi -\frac{b^{2}}{\mu }\right)
		\left\Vert \phi \right\Vert ^{2}-\alpha \left\Vert \phi _{x}\right\Vert
		^{2}+2J\left\Vert \phi _{t}\right\Vert ^{2} \\
		& -2\beta \left\langle \phi ,\theta _{x}\right\rangle -\frac{\mu \beta }{b}%
		\left\langle \theta _{x},u_{x}\right\rangle -\dfrac{\rho \xi }{b}%
		\left\langle u_{t},\int_{0}^{x}\phi _{t}(y)dy\right\rangle +\left( \frac{%
			J\mu }{b}-\frac{\alpha \rho }{b}\right) \left\langle u_{tx},\phi
		_{t}\right\rangle\\
		& -\frac{\rho \beta \mu g(0)}{b\gamma _{g}}\left\langle \theta
		_{x},u_{x}\right\rangle -\frac{\rho \beta \sqrt{\mu }}{b\gamma _{g}}%
		\int_{0}^{+\infty }\kappa ^{\prime }(s)\left\langle \eta _{x},\sqrt{\mu }%
		u_{x}+\frac{b}{\sqrt{\mu }}\phi \right\rangle ds \\
		& +\frac{\rho c\beta \mu }{b\gamma _{g}}\left\langle \theta
		_{t},u_{t}\right\rangle +\frac{\rho \beta ^{2}\mu }{b\gamma _{g}}%
		\left\langle \phi _{xt},u_{t}\right\rangle -\frac{\rho \beta }{\gamma _{g}}%
		\int_{0}^{+\infty }\kappa (s)\left\langle \eta _{x}(s),\phi _{t}\right\rangle
		ds.\\
		&+\frac{\rho \beta g(0)}{\gamma _{g}}\left\langle
		\theta ,\phi _{x}\right\rangle -\frac{\mu c\beta \rho }{b\gamma _{g}}\left\langle \theta
		_{t},u_{t}\right\rangle +\frac{\beta \mu ^{2}c}{b\gamma _{g}}\left\langle
		\theta _{x},u_{x}\right\rangle -\frac{\mu c\beta }{\gamma _{g}}\left\langle
		\theta ,\phi _{x}\right\rangle ,
	\end{align*}%
	then,%
	\begin{align*}
		\frac{d}{dt}F_{3}\left( t\right) =&-\left\Vert \sqrt{\mu }u_{x}+\frac{b}{%
			\sqrt{\mu }}\phi \right\Vert ^{2}-\left( \xi -\frac{b^{2}}{\mu }\right)
		\left\Vert \phi \right\Vert ^{2}-\alpha \left\Vert \phi _{x}\right\Vert
		^{2}+2J\left\Vert \phi _{t}\right\Vert ^{2} \\
		& +\frac{\mu }{b}\left( \frac{\mu c\beta }{\gamma _{g}}-\frac{\rho \beta g(0)%
		}{\gamma _{g}}-\beta \right) \left\langle \theta _{x},u_{x}\right\rangle -%
		\dfrac{\rho \xi }{b}\left\langle u_{t},\int_{0}^{x}\phi
		_{t}(y)dy\right\rangle \\
		& +\left( \frac{\rho \beta g(0)}{\gamma _{g}}-\frac{\mu c\beta }{\gamma _{g}}%
		+2\beta \right) \left\langle \theta ,\phi _{x}\right\rangle -\frac{\mu
			\alpha }{b}\chi _{g}\left\langle \phi _{t},u_{tx}\right\rangle \\
		& -\frac{\rho \beta \sqrt{\mu }}{b\gamma _{g}}\int_{0}^{+\infty }\kappa
		^{\prime }(s)\left\langle \eta _{x},\sqrt{\mu }u_{x}+\frac{b}{\sqrt{\mu }}%
		\phi \right\rangle ds-\frac{\rho \beta }{\gamma _{g}}\int_{0}^{+\infty
		}\kappa (s)\left\langle \eta _{x}(s),\phi _{t}\right\rangle ds,
	\end{align*}%
	Taking into account that $\chi _{g}=0,$ we obtain%
	\begin{align}
		\frac{d}{dt}F_{3}\left( t\right) =&-\left\Vert \sqrt{\mu }u_{x}+\frac{b}{%
			\sqrt{\mu }}\phi \right\Vert ^{2}-\left( \xi -\frac{b^{2}}{\mu }\right)
		\left\Vert \phi \right\Vert ^{2}-\alpha \left\Vert \phi _{x}\right\Vert
		^{2}+2J\left\Vert \phi _{t}\right\Vert ^{2}  \notag \\
		& -\dfrac{\rho \xi }{b}\left\langle u_{t},\int_{0}^{x}\phi
		_{t}(y)dy\right\rangle -\frac{\rho \beta \sqrt{\mu }}{b\gamma _{g}}%
		\int_{0}^{+\infty }\kappa ^{\prime }(s)\left\langle \eta _{x},\sqrt{\mu }%
		u_{x}+\frac{b}{\sqrt{\mu }}\phi \right\rangle ds  \notag \\
		& +\beta \left\langle \theta ,\phi _{x}\right\rangle -\frac{\rho \beta }{%
			\gamma _{g}}\int_{0}^{+\infty }\kappa (s)\left\langle \eta _{x}(s),\phi
		_{t}\right\rangle ds.  \label{32b}
	\end{align}%
	Cauchy-Schwarz and Young's inequalities yield%
	\begin{equation*}  
		-\frac{\rho \beta \sqrt{\mu }}{b\gamma _{g}}\int_{0}^{+\infty }\kappa
		^{\prime }(s)\left\langle \eta _{x},\sqrt{\mu }u_{x}+\frac{b}{\sqrt{\mu }}%
		\phi \right\rangle ds
	\end{equation*}
\begin{equation}\label{32c}
	\leq \frac{1}{2}\left\Vert \sqrt{\mu }u_{x}+\frac{b}{%
			\sqrt{\mu }}\phi \right\Vert ^{2}-M\int_{0}^{+\infty }\kappa ^{^{\prime
		}}(s)\left\Vert \eta _{x}\right\Vert ^{2}ds.
	\end{equation}%
	Similarly, for any $\varepsilon>0$,%
	\begin{equation}  \label{32d}
		\dfrac{\rho \xi }{b}\left\langle u_{t},\int_{0}^{x}\phi
		_{t}(y)dy\right\rangle \leq \varepsilon \left\Vert u_{t}\right\Vert ^{2}+%
		\dfrac{M}{\varepsilon }\left\Vert \phi _{t}\right\Vert ^{2},
	\end{equation}%
	\begin{equation}  \label{32e}
		\beta \left\langle \theta ,\phi _{x}\right\rangle \leq \frac{\alpha }{2}%
		\left\Vert \phi _{x}\right\Vert ^{2}+M\left\Vert \theta \right\Vert ^{2}.
	\end{equation}%
	and by the use of (\ref{F}), for any $\varepsilon_{3}>0$ we have 
	\begin{align}
		\frac{\rho \beta }{\gamma _{g}}\int_{0}^{+\infty }\kappa (s)\left\langle \eta
		_{x}(s),\phi _{t}\right\rangle ds& \leq \frac{1}{\varepsilon _{3}}\left\Vert
		\phi _{t}\right\Vert ^{2}+\varepsilon _{3}M\left\Vert \eta \right\Vert _{%
			\mathcal{V}}^{2}  \notag \\
		& \leq \frac{1}{\varepsilon _{3}}\left\Vert \phi _{t}\right\Vert
		^{2}-\varepsilon _{3}M\int_{0}^{+\infty }\kappa ^{^{\prime }}(s)\left\Vert
		\eta _{x}\right\Vert ^{2}ds.  \label{32f}
	\end{align}%
	Therefore, (\ref{35}) follows from (\ref{32b})-(\ref{32f}). 
\end{proof}

\begin{lemma}
	For $(u,\phi ,\theta ,\eta )$ solution of (\ref{1}), the functional 
	\begin{equation*}
		F_{4}(t):=-\rho \left\langle u_{t},u\right\rangle -\dfrac{b\rho }{\mu }%
		\left\langle u_{t},\int_{0}^{x}\phi (y)dy\right\rangle
	\end{equation*}%
	satisfies 
	\begin{equation}
		\frac{d}{dt}F_{4}(t)\leq -\dfrac{\rho }{2}\left\Vert u_{t}\right\Vert
		^{2}+\left\Vert \sqrt{\mu }u_{x}+\dfrac{b}{\sqrt{\mu }}\phi \right\Vert
		^{2}+M\left\Vert \phi _{t}\right\Vert ^{2},  \label{40}
	\end{equation}%
	where $M$ is a positive constant.
\end{lemma}

\begin{proof}
	By differentiating $F_{4}(t)$ and using (\ref{1}), we arrive at 
	\begin{align*}
		\frac{d}{dt}F_{4}(t)=&-\rho \left\Vert u_{t}\right\Vert ^{2}-\left\langle \mu
		u_{xx}+b\phi _{x},u\right\rangle\\
		& -\dfrac{b}{\mu }\left\langle \mu
		u_{xx}+b\phi _{x},\int_{0}^{x}\phi (y)dy\right\rangle -\dfrac{b\rho }{\mu }%
		\left\langle u_{t},\int_{0}^{x}\phi _{t}(y)dy\right\rangle ,\\
		=&-\rho \left\Vert u_{t}\right\Vert ^{2}+\mu \left\Vert u_{x}\right\Vert
		^{2}+2b\left\langle u_{x},\phi \right\rangle +\dfrac{b^{2}}{\mu }\left\Vert
		\phi \right\Vert ^{2}-\dfrac{b\rho }{\mu }\left\langle
		u_{t},\int_{0}^{x}\phi _{t}(y)dy\right\rangle , \\
		=&-\rho \left\Vert u_{t}\right\Vert ^{2}+\left\Vert \sqrt{\mu }u_{x}+\dfrac{%
			b}{\sqrt{\mu }}\phi \right\Vert ^{2}-\dfrac{b\rho }{\mu }\left\langle
		u_{t},\int_{0}^{x}\phi _{t}(y)dy\right\rangle .
	\end{align*}%
	Young and Cauchy Schwarz inequalities, give%
	\begin{equation*}
		\dfrac{b\rho }{a}\left\langle u_{t},\int_{0}^{x}\phi _{t}(y)dy\right\rangle
		\leq \dfrac{\rho }{2}\left\Vert u_{t}\right\Vert ^{2}+M\left\Vert \phi
		_{t}\right\Vert ^{2}
	\end{equation*}%
	and (\ref{40}) follows immediately.
\end{proof}

At this state, we introduce the Lyapunov functional%
\begin{equation*}
	\mathscr{L}(t)=NE\left( t\right) +N_{1}F_{1}\left( t\right) +\frac{1}{%
		\varepsilon _{2}}F_{2}\left( t\right) +\frac{\rho }{4\varepsilon _{3}}%
	F_{3}\left( t\right) +F_{4}(t),
\end{equation*}%
where, $N$ and $N_{1}$ are positive constants the will be determine later.

By differentiating $\mathscr{L}(t)$ and using (\ref{25}),(\ref{30}),(\ref{35}%
),(\ref{40}) we get

\begin{align*}
	\frac{d}{dt}\mathscr{L}(t) \leq &-\left( \frac{\rho }{8\varepsilon _{3}}%
	-2\right) \left\Vert \sqrt{\mu }u_{x}+\frac{b}{\sqrt{\mu }}\phi \right\Vert
	^{2} \\
	& -\left[ \left( \xi -\frac{b^{2}}{\mu }\right) \frac{\rho }{4\varepsilon
		_{3}}-1\right] \left\Vert \phi \right\Vert ^{2}-\frac{\rho }{4}\left\Vert
	u_{t}\right\Vert ^{2}-\left( \frac{\rho \alpha }{8\varepsilon _{3}}-1\right)
	\left\Vert \phi _{x}\right\Vert ^{2} \\
	& -\left( \frac{J}{2\varepsilon _{2}}-\varepsilon _{1}N_{1}-M\left( \frac{%
		\rho }{4\varepsilon _{3}^{2}}+\frac{\rho }{4\varepsilon _{3}}+1\right)
	\right) \left\Vert \phi _{t}\right\Vert ^{2} \\
	& -\left( cN_{1}-M\left( 1+\frac{1}{\varepsilon _{2}}\right) \frac{1}{%
		\varepsilon _{2}}-M\frac{\rho }{4\varepsilon _{3}}\right) \left\Vert \theta
	\right\Vert ^{2}-N_{1}\left\Vert \eta \right\Vert _{\mathcal{V}}^{2} \\
	& +\left( N-MN_{1}\left( 1+\frac{1}{\varepsilon _{1}}\right) -\frac{M}{%
		\varepsilon _{2}}-\frac{M\rho }{4\varepsilon _{3}}\left( 1+\varepsilon
	_{3}\right) \right) \int_{0}^{+\infty }\kappa ^{\prime }(s)\left\Vert \eta
	_{x}\right\Vert ^{2}ds
\end{align*}%
First we choose $\varepsilon _{3}$ small enough such that 
\begin{equation*}
	\varepsilon _{3}\leq \min \left\{ \frac{\rho }{16},\frac{\rho \alpha }{8}%
	,\left( \xi -\frac{b^{2}}{\mu }\right) \frac{\rho }{4}\right\} .
\end{equation*}%
Next, we choose%
\begin{equation*}
	\varepsilon _{2}=\frac{2J\varepsilon _{3}^{2}}{M\left[ \rho \varepsilon
		_{3}+\rho +4\varepsilon _{3}^{2}\right] },
\end{equation*}%
to get%
\begin{equation*}
	\frac{J}{2\varepsilon _{2}}-M\left( \frac{\rho }{4\varepsilon _{3}^{2}}+%
	\frac{\rho }{4\varepsilon _{3}}+1\right) =M\left( \frac{\rho }{4\varepsilon
		_{3}^{2}}+\frac{\rho }{4\varepsilon _{3}}+1\right) .
\end{equation*}%
After that, we pick $N_{1}$ large such that%
\begin{equation*}
	cN_{1}>M\left[ \left( 1+\frac{1}{\varepsilon _{2}}\right) \frac{1}{%
		\varepsilon _{2}}+\frac{\rho }{4\varepsilon _{3}}\right] .
\end{equation*}%
Next, we choose $\varepsilon _{1}$ small enough such that%
\begin{equation*}
	M\left( \frac{\rho }{4\varepsilon _{3}^{2}}+\frac{\rho }{4\varepsilon _{3}}%
	+1\right) >\varepsilon _{1}N_{1}.
\end{equation*}%
Thus, there exist $\varpi ,\lambda >0,$ such that%
\begin{eqnarray}
	\frac{d}{dt}\mathscr{L}(t) &\leq &-\varpi \left(\left\Vert \phi \right\Vert
	^{2}+\left\Vert u_{t}\right\Vert ^{2}+\left\Vert \phi _{x}\right\Vert
	^{2}+\left\Vert \phi _{t}\right\Vert ^{2}+\left\Vert \theta \right\Vert
	^{2}+\left\Vert \eta \right\Vert ^{2}\right)  \notag \\
	&&-\varpi  \left\Vert \sqrt{\mu }u_{x}+%
	\frac{b}{\sqrt{\mu }}\phi \right\Vert ^{2}
	+\left( N-\lambda \right) \int_{0}^{\infty }\kappa ^{^{\prime
	}}(s)\left\Vert \eta _{x}\right\Vert ^{2}ds.  \label{34}
\end{eqnarray}%
On the other hand, let 
\begin{equation*}
	\mathscr{K}\left( t\right) :=\mathscr{L}\left( t\right) -NE\left( t\right)
	=N_{1}F_{1}\left( t\right) +\frac{1}{\varepsilon _{2}}F_{2}\left( t\right) +%
	\frac{\rho }{4\varepsilon _{3}}F_{3}\left( t\right) +F_{4}(t),
\end{equation*}%
then,%
\begin{align*}
	\left\vert \mathscr{K}\left( t\right) \right\vert \leq &C\left[ \left\vert
	\int_{0}^{+\infty }\kappa (s)\left\langle \theta (t),\eta
	^{t}(s)\right\rangle ds\right\vert +\left\vert \left\langle \theta
	,\int_{0}^{x}\phi _{t}(y)dy\right\rangle \right\vert\right]\\
	 +&C\left[\left\vert
	\left\langle \phi _{x},u_{t}\right\rangle \right\vert +\left\vert
	\left\langle \phi _{t},u_{x}\right\rangle \right\vert +\left\vert
	\left\langle \phi ,\phi _{t}\right\rangle \right\vert + \left\vert \left\langle u_{t},u\right\rangle \right\vert
	+\left\vert \left\langle u_{t},\int_{0}^{x}\phi (y)dy\right\rangle
	\right\vert \right]\\
	+&C\left[ \left\vert \left\langle \int_{0}^{\infty }\kappa (s)\eta
	_{x}(s)ds,\sqrt{\mu }u_{x}+\frac{b}{\sqrt{\mu }}\phi \right\rangle
	\right\vert +\left\vert \left\langle \theta ,u_{t}\right\rangle \right\vert %
	\right]\\
\end{align*}%
Therefore, Cauchy-Schwarz and Young's inequalities yield%
\begin{equation*}
	\left\vert \mathscr{L}\left( t\right) -NE\left( t\right) \right\vert \leq
	CE\left( t\right) ,
\end{equation*}%
that is%
\begin{equation*}
	\left( N-C\right) E\left( t\right) \leq \mathscr{L}\left( t\right) \leq
	\left( N-C\right) E\left( t\right) .
\end{equation*}%
Now, we choose $N$ large enough such that $N>\max \left\{ C,\lambda \right\}
.$ Then, $\mathscr{L}\left( t\right) \sim E\left( t\right) $ and from (\ref%
{34}) we infer that there exist $\gamma >0$ such that 
\begin{equation*}
	\frac{d}{dt}\mathscr{L}(t)\leq -\gamma E\left( t\right) ,\;t\geq 0.
\end{equation*}%
Using the fact that $L\left( t\right) \sim E\left( t\right) ,$ we get%
\begin{equation*}
	\frac{d}{dt}\mathscr{L}(t)\leq -\omega L\left( t\right) ,\;t\geq 0,
\end{equation*}%
where, $\omega >0.$ Therefore,%
\begin{equation*}
	\mathscr{L}\left( t\right) \leq \mathscr{L}\left( 0\right) e^{-\omega
		t},\;t\geq 0.
\end{equation*}%
Again, the fact that $\mathscr{L}\left( t\right) \sim E\left( t\right) ,$
infer that%
\begin{equation*}
	E\left( t\right) \leq \sigma e^{-\omega t},\;t\geq 0,
\end{equation*}%
which completes the proof of Theorem \ref{TH3}.

\section{Lack of exponential stability}

In this section we examine the cases when $\gamma _{g}=0$ or $\chi _{g}\neq
0 $ and prove that the solution $(u,\phi ,\theta ,\eta )$ loses its
exponential stability, that is:

\begin{theorem}
	\label{TH4} Suppose that $\gamma _{g}=0,$ or $\gamma _{g}\neq 0$ and $\chi
	_{g}\neq 0,$ then the solution $\left( u,\phi ,\theta ,\eta \right) $ can
	not be exponentially stable.
\end{theorem}
The proof is based on the following theorem
\begin{theorem}
	Let \ $\mathcal{B}$ (unbounded operator) be the infinitesimal generator of \
	a semigroup of contractions $S(t)=e^{\mathcal{B}t}.$ Then $S\left( t\right) $
	is exponentially stable if and only if there exists a positive constant $m$
	such that%
	\begin{equation}
		\underset{\lambda \in \mathbb{R}}{\inf }\Vert (i\lambda I-\mathcal{B})U\Vert
		\geq m\Vert U\Vert ,~\forall U\in D(\mathcal{B}).
	\end{equation}
\end{theorem}
\begin{proof}[Proof of Theorem \ref{TH4}]
	
	Taking $F_{n}=(0,\frac{\sin (n\pi )}{\rho },0,0,0,0)$, then 
	\begin{equation*}
		\Vert F_{n}\Vert _{\mathscr{H}}=\sqrt{\frac{\pi }{2\rho }}.
	\end{equation*}%
	Let $U_{n}\in D(\mathscr{A})$ be the solution of $(i\lambda _{n}I+\mathscr{A}%
	)U_{n}=F_{n}$, then 
	\begin{equation}
		\left\{ 
		\begin{array}{rl}
			i\lambda _{n}u_{n}-v_{n}= & 0 \\ 
			i\rho \lambda _{n}v_{n}-\mu u_{nxx}-b\phi _{nx}= & \sin (n\pi ), \\ 
			i\lambda _{n}\phi _{n}-\psi _{n}= & 0, \\ 
			iJ\lambda _{n}\psi _{n}-\alpha \phi _{nxx}+\xi \phi _{n}+bu_{nx}+\beta
			\theta _{nx}= & 0, \\ 
			ic\lambda _{n}\theta _{n}+\beta \psi _{nx}-\displaystyle\int_{0}^{+\infty
			}\kappa (s)\eta _{nxx}(s)ds= & 0, \\ 
			i\lambda _{n}\eta _{n}-\theta _{n}+\eta _{ns}= & 0%
		\end{array}%
		\right.   \label{41}
	\end{equation}%
	Simplifying $v_{n},\psi _{n}$ from (\ref{41}), we obtain 
	\begin{equation}
		\left\{ 
		\begin{array}{rl}
			\rho \lambda _{n}^{2}u_{n}+\mu u_{nxx}+b\phi _{nx}= & -\sin (n\pi ), \\ 
			(J\lambda _{n}^{2}-\xi )\phi _{n}+\alpha \phi _{nxx}-bu_{nx}-\beta \theta
			_{nx}= & 0, \\ 
			ic\lambda _{n}\theta _{n}+i\lambda _{n}\beta \phi _{nx}-\displaystyle%
			\int_{0}^{+\infty }\kappa (s)\eta _{nxx}(s)ds= & 0, \\ 
			i\lambda _{n}\eta _{n}+\theta _{n}-\eta _{ns}= & 0%
		\end{array}%
		\right.   \label{42}
	\end{equation}%
	Taking into account the boundary conditions (\ref{3}), we set 
	\begin{equation}
		\left\{ 
		\begin{array}{l}
			u_{n}=A_{n}\sin (nx), \\ 
			\phi _{n}=B_{n}\cos (nx), \\ 
			\theta _{n}=C_{n}\sin (nx), \\ 
			\eta _{n}=D_{n}(s)\sin (nx),%
		\end{array}%
		\right.   \label{42'}
	\end{equation}%
	where, $A_{n},B_{n},C_{n}\in \mathbb{C}$ and $D_{n}:\mathbb{R}%
	_{+}\rightarrow \mathbb{C}$ such that $\displaystyle\int_{0}^{\infty }\kappa
	(s)|D_{n}(s)|^{2}ds<0$, and $D_{n}(0)=0.$ Plugging (\ref{42'}) into (\ref{42}%
	), \ we infer
	\begin{equation}
		\left\{ 
		\begin{array}{rl}
			\left( \rho \lambda _{n}^{2}-\mu n^{2}\right) A_{n}-bnB_{n}= & 1, \\ 
			(J\lambda _{n}^{2}-\xi -\alpha n^{2})B_{n}-bnA_{n}-\beta nC_{n}= & 0, \\ 
			ic\lambda _{n}C_{n}-i\beta n\lambda _{n}B_{n}+n^{2}\displaystyle%
			\int_{0}^{+\infty }\kappa (s)D_{n}(s)ds= & 0, \\ 
			i\lambda _{n}D_{n}\left( s\right) +C_{n}-D_{n}^{\prime }\left( s\right) = & 
			0.%
		\end{array}%
		\right.   \label{43}
	\end{equation}%
	Solving (\ref{43})$_{4}$, we get%
	\begin{equation*}
		D_{n}\left( s\right) =\frac{C_{n}}{i\lambda _{n}}\left( 1-e^{-i\lambda
			_{n}s}\right) .
	\end{equation*}%
	The substitution of $D_{n}\left( s\right) $ into (\ref{43})$_{3},$ entails
	
	\begin{equation*}
		\begin{array}{rl}
			\left( \rho \lambda _{n}^{2}-\mu n^{2}\right) A_{n}-bnB_{n} & =-1 \\ 
			(J\lambda _{n}^{2}-\xi -\alpha n^{2})B_{n}-bnA_{n}-\beta nC_{n} & =0 \\ 
			\beta n\lambda _{n}^{2}B_{n}-\left[ c\lambda _{n}^{2}-n^{2}\left( g\left(
			0\right) -\widehat{\kappa }\left( \lambda _{n}\right) \right) \right] C_{n}
			& =0%
		\end{array}%
	\end{equation*}%
	which can be written as%
	\begin{equation}
		\left( 
		\begin{array}{ccc}
			p_{1}\left( \lambda \right)  & -bn & 0 \\ 
			-bn & p_{2}\left( \lambda \right)  & -\beta n \\ 
			0 & \beta n\lambda _{n}^{2} & p_{3}\left( \lambda \right) 
		\end{array}%
		\right) \left( 
		\begin{array}{c}
			A_{n} \\ 
			B_{n} \\ 
			C_{n}%
		\end{array}%
		\right) =\left( 
		\begin{array}{c}
			-1 \\ 
			0 \\ 
			0%
		\end{array}%
		\right) 
	\end{equation}%
	where,
	
	\begin{equation*}
		p_{1}\left( \lambda \right) =\left( \rho \lambda _{n}^{2}-\mu n^{2}\right) ,%
		\text{ \ \ }p_{2}\left( \lambda \right) =(J\lambda _{n}^{2}-\xi -\alpha
		n^{2}),\text{ \ \ }p_{3}\left( \lambda \right) =-c\lambda
		_{n}^{2}+n^{2}\left( g\left( 0\right) -\widehat{\kappa }\left( \lambda
		_{n}\right) \right) .
	\end{equation*}%
	It follows that%
	\begin{equation*}
		A_{n}=\frac{-K_{1}}{p_{1}\left( \lambda _{n}\right) K_{1}-K_{2}},
	\end{equation*}%
	where%
	\begin{equation*}
		K_{1}=p_{2}\left( \lambda _{n}\right) p_{3}\left( \lambda _{n}\right)
		+\left( \beta n\lambda _{n}\right) ^{2},\;K_{2}=\left( bn\right)
		^{2}p_{3}\left( \lambda _{n}\right) .
	\end{equation*}%
	Taking $\lambda_{n}$ such that $p_{1}\left( \lambda _{n}\right) =0,$ that is $n^{2}=\frac{\rho }{\mu 
	}\lambda _{n}^{2},$ then
	
	\begin{equation}
		A_{n}=\frac{K_{1}}{K_{2}}.  \label{46}
	\end{equation}%
	where,
	
	\begin{equation*}
		K_{1}=\frac{\alpha }{\mu }\left( \frac{\rho }{\mu }-\frac{J}{\alpha }+\frac{%
			\xi }{\alpha \lambda _{n}^{2}}\right) \left[ c\mu -\rho g\left( 0\right)
		+\rho \widehat{\kappa }\left( \lambda _{n}\right) \right] \lambda _{n}^{4}+%
		\frac{\rho \beta ^{2}}{\mu }\lambda _{n}^{4},
	\end{equation*}%
	as $\frac{\xi }{\alpha \lambda _{n}^{2}}\longrightarrow 0,$ then%
	\begin{equation}\label{47}
		K_{1}\simeq \frac{\alpha }{\mu }\left[ \left( \frac{\rho }{\mu }-\frac{J}{%
			\alpha }\right) \gamma _{g}+\frac{\rho \beta ^{2}}{\alpha }+\left( \frac{%
			\rho }{\mu }-\frac{J}{\alpha }\right) \rho \widehat{\kappa }\left( \lambda
		_{n}\right) \right] \lambda _{n}^{4} 
	\end{equation}%
	and%
	\begin{eqnarray}
		K_{2} &=&\frac{b^{2}\rho }{\mu ^{2}}\lambda _{n}^{4}\left( -c\mu +\rho
		g\left( 0\right) -\rho \widehat{\kappa }\left( \lambda _{n}\right) \right)  
		\notag \\
		&=&-\frac{b^{2}\rho }{\mu ^{2}}\left( \gamma _{g}+\rho \widehat{\kappa }%
		\left( \lambda _{n}\right) \right) \lambda _{n}^{4}.  \label{48}
	\end{eqnarray}%
	Substituting (\ref{47}) and (\ref{48}) into (\ref{46}) we get
	
	\begin{equation*}
		A_{n}\simeq -\frac{\alpha \mu \left[ \chi _{g}\gamma _{g}+\left( \frac{\rho 
			}{\mu }-\frac{J}{\alpha }\right) \rho \widehat{\kappa }\left( \lambda
			_{n}\right) \right] }{b^{2}\rho \left( \gamma _{g}+\rho \widehat{\kappa }%
			\left( \lambda _{n}\right) \right) }.
	\end{equation*}
	\begin{itemize}
		\item If $\dfrac{\rho }{\mu }-\dfrac{J}{\alpha }=0$ and $\gamma _{g}\neq 0,$  since  $\underset{n\longrightarrow \infty }{\lim }\widehat{\kappa }\left(
		\lambda _{n}\right) =0,$ we have
		\begin{equation*}
			A_{n}\simeq \frac{\beta ^{2}\mu }{-b^{2}\left( \gamma _{g}+\rho \widehat{%
					\kappa }\left( \lambda _{n}\right) \right) }\longrightarrow \frac{\beta
				^{2}\mu }{-b^{2}\gamma _{g}}=A\neq 0,
		\end{equation*}%
		where $A$ is a constant.
		
		\item If $\dfrac{\rho }{\mu }-\dfrac{J}{\alpha }\neq 0$ and $\gamma _{g}=0,$ then%
		\begin{equation*}
			A_{n}\simeq -\frac{\alpha \mu \left( \frac{\rho }{\mu }-\frac{J}{\alpha }%
				\right) }{b^{2}\rho }=A\neq 0.
		\end{equation*}%
		\item If $\gamma _{g}=0$ and $\dfrac{\rho }{\mu }-\dfrac{J}{\alpha }=0,$ then 
		\begin{equation*}
			K_{1}=\frac{\rho \xi }{\mu }\widehat{\kappa }\left( \lambda _{n}\right)
			\lambda _{n}^{2}+\frac{\rho \beta ^{2}}{\mu }\lambda _{n}^{4},\;K_{2}=-\frac{%
				b^{2}\rho ^{2}}{\mu ^{2}}\widehat{\kappa }\left( \lambda _{n}\right) \lambda
			_{n}^{4}
		\end{equation*}%
		then%
		\begin{equation*}
			A_{n}\simeq \frac{\mu \beta ^{2}}{-b^{2}\rho \widehat{\kappa }\left( \lambda
				_{n}\right) }\longrightarrow \infty .
		\end{equation*},
		since $\underset{n\longrightarrow \infty }{\lim }\widehat{\kappa }\left(
		\lambda _{n}\right) =0.$
		\item If $\gamma _{g}\neq 0$, and $\chi _{g}\neq 0,$ then \ 
		\begin{equation*}
			A_{n}\simeq -\frac{\alpha \mu \left[ \chi _{g}+\left( \frac{\rho }{\mu }-%
				\frac{J}{\alpha }\right) \rho \frac{\widehat{\kappa }\left( \lambda
					_{n}\right) }{\gamma _{g}}\right] }{b^{2}\rho \left( 1+\rho \frac{\widehat{%
						\kappa }\left( \lambda _{n}\right) }{\gamma _{g}}\right) },
		\end{equation*}%
		and%
		\begin{equation*}
			A_{n}\longrightarrow -\frac{\alpha \mu \chi _{g}}{b^{2}\rho }=A\neq 0.
		\end{equation*}%
		
	\end{itemize}
	Therefore, if $A_{n}\longrightarrow A,$ we have
	\begin{align*}
		\left\Vert U_{n}\right\Vert ^{2}\geq  & \rho \int_{0}^{\pi }\left\vert
		v_{n}\right\vert ^{2}dx =\rho \int_{0}^{\pi }\left\vert \lambda
		_{n}u_{n}\right\vert ^{2}dx=\rho \int_{0}^{\pi }\left\vert \lambda _{n}A\sin
		nx\right\vert ^{2}dx, \\
		\geq & \rho \left\vert \lambda _{n}A\right\vert \int_{0}^{\pi }\left\vert
		\sin nx\right\vert ^{2}dx =\frac{\rho \pi \left\vert \lambda
			_{n}A\right\vert }{2}\longrightarrow \infty ,
	\end{align*}
	and if $A_{n}\longrightarrow \infty ,$ we have 
	\begin{equation*}
		\left\Vert U_{n}\right\Vert ^{2}\geq   \mu \int_{0}^{\pi }\left\vert
		u_{n}\right\vert ^{2}dx=\mu \int_{0}^{\pi }\left\vert A_{n}u_{n}\right\vert
		^{2}dx\longrightarrow \infty ,
	\end{equation*}%
	which completes the proof of Theorem \ref{TH4}.
	\end{proof}
	\section*{Declarations}

	\begin{itemize}
		\item {\bf Funding}: This work is funded by DGRSDT Algeria PRFU N° C00L03UN390120220004.
		\item {\bf Conflict of interest}: There is no conflict of interest.		
		\item {\bf Authors' contributions}: All the authors are contributed equally to this work
		\item {\bf Data Availability Statements}: Data sharing not applicable to this article as no datasets were generated or analysed during the current study since it is a theoretical study.
	
	\end{itemize}

\end{document}